\documentclass[12pt]{amsart}
\usepackage[utf8]{inputenc}

\usepackage{BAstyle}
\usepackage[margin=1in]{geometry}
\newtheorem{remark}[theorem]{Remark}
\theoremstyle{plain}
\newtheorem{data}[theorem]{Data Analysis}

\title{Power Savings for Counting (Twisted) Abelian Extensions of Number Fields}
\author{Brandon Alberts}
\address{Eastern Michigan University}
\email{balbert1@emich.edu}

\begin{document}

\begin{abstract}
We prove significant power savings for the error term when counting abelian extensions of number fields (as well as the twisted version of these results for nontrivial Galois modules). In some cases over $\Q$, these results reveal lower order terms following the same structure as the main term that were not previously known. Assuming the generalized Lindel\"of hypothesis for Hecke $L$-functions, we prove square root power savings for the error compared to the order of the main term.
\end{abstract}

\maketitle

\setcounter{tocdepth}{1}
\tableofcontents

\pagebreak
\section{Introduction}

The primary focus in number field counting problems has been on identifying the main term of the asymptotic growth rate of number fields extensions with given properties. Given a number field $K$ with absolute Galois group $G_K$ and a transitive subgroup $G\subseteq S_n$, Malle \cite{malle2002,malle2004} studied the growth rate of the function
\[
\#\Surj(G_K,G;X) = \#\{\pi\in \Surj(G_K,G) : |\disc(\pi)|\le X\},
\]
where $\disc(\pi)$ is the discriminant of the field fixed by $\pi^{-1}(\Stab_G(1)) \le G_K$ and $|\cdot|$ denotes the norm down to $\Q$. By the Galois correspondence, this is equivalent (up to a constant factor) to the function counting degree $n$ extensions $L/K$ with Galois group $G$ and bounded discriminant. Malle predicts that this function grows asymptotic to
\[
c X^{1/a(G)}(\log X)^{b(K,G)-1}.
\]
as $X$ tends towards infinity, where $c>0$ is some positive constant, $a(G) = \min_{g\ne 1} \ind(g)$, $\ind(g) = n - \#\{\text{orbits of }g\}$, and $b(K,G)$ is the number of orbits of non-identity conjugacy classes of $G$ under the cyclotomic action of minimal index.

\begin{remark}
    It is known that, for certain groups, $b(K,G)$ is the incorrect invariant to place in the power of the logarithm. There does exist a modified conjecture which attempts to correct for this error due to T\"urkelli \cite{alberts2021,turkelli2015}, but we will not detail it in this paper. For all the groups we consider in this paper, Malle's conjecture is known to be correct with $b(K,G)-1$ as the correct power of $\log X$.
\end{remark}

In virtually every known case of Malle's conjecture, the methods used can do a little better than the main term and will instead prove something of the form
\[
\#\Surj(G_K,G;X) = X^{1/a(G)}P(\log X) + o(X^{1/a(G)}),
\]
where $P$ is a polynomial of degree $b(K,G)$ with positive leading coefficient. This quality of error term is not always mentioned explicitly, but often readily follows from the author's methods (most often a Tauberian theorem). We summarize the current progress towards Malle's conjecture in Subsection \ref{subsec:history}. The relevant case for us is that Malle's conjecture is proven for all abelian groups in \cite{wright1989}.

The goal of this paper is to prove power saving error bounds for $\#\Surj(G_K,G;X)$ when $G$ is an abelian group, of the form $O(X^{1/a(G) - \delta})$ for some $\delta > 0$. Currently, only the groups $S_3$ \cite{taniguchi-thorne2013}, $S_3\times A$ for $A$ abelian \cite{jwang2017}, and $D_4$ have explicit power saving error bounds rigorously treated in the literature. The existence of power savings for abelian groups is proven in \cite{frei-loughran-newton2018}, but is not made explicit. Cohen--Diaz y Diaz--Olivier remark in \cite{cohen-diaz-y-diaz-olivier2006} that standard contour integration techniques produce an explicit power saving error bound for counting $C_4$-extensions, although they could not find such a result in the literature. Similar reasoning can be applied to other small abelian groups like $C_2$ and $C_3$, but explicit power saving error bounds for counting abelian extensions have not been treated in general.

The power savings for $S_3$-, $S_3\times A$-, and $C_4$-extensions over $\Q$ also reveal a secondary term in the asymptotic growth rate. A secondary goal of this paper is to prove some results on the existence of lower order terms for counting abelian extensions.

\subsection{Power Saving Bounds for Counting Abelian Extensions}
We give a rigorous treatment of the ``standard contour integration technique". We first prove a meromorphic continuation for the generating Dirichlet series:

\begin{theorem}\label{thm:generating_mero}
    Let $K$ be a number field and $G$ an abelian group. The generating Dirichlet series for $G$-extensions of $K$ ordered by discriminant is meromorphic on the right halfplane ${\rm Re}(s) > 0$.

    Any poles with real part larger than $\frac{1}{2a(G)}$ occur at $s=\frac{1}{d}$ for each $d\in \ind(G-\{1\})$ of order \textbf{at most} $b_d(K,G)$, where
    \[
    b_d(K,G) = \#\{\text{orbits }o\text{ of }G\text{ under }\chi\text{ such that }\ind(o)=d\}.
    \]
    Here, $\chi:G_K\to \hat{Z}^\times$ is the cyclotomic character and induces a Galois action on $G$ by $x.g = g^{\chi(g)}$.
\end{theorem}

Each pole is expected to contribute a term to the asymptotic growth rate. This is usually achieved by a Tauberian theorem, which is often proven by shifting a contour integral past the poles. This process leaves behind some error term, and we define a new invariant to capture the power saving bound for the error term in $\#\Surj(G_K,G;X)$.

\begin{definition}\label{def:theta}
    Let $K$ be a number field, $G\subseteq S_n$ a transitive group, and $D_{K,G}(s)$ the generating series of $G$-extensions of $K$. Define $\theta(K,G)$ to be the infimum of all real numbers $\theta$ for which
    \begin{enumerate}[(a)]
        \item $D_{K,G}(s)$ has a meromorphic continuation to ${\rm Re}(s) > \theta$ with at most finitely many poles $s_1,s_2,...,s_r$, and
        \item The number of $G$-extensions of $K$ satisfies the bound
        \[
            \#\Surj(G_K,G;X) - \sum_{j=1}^r \underset{s=s_j}{\rm Res}\left(D_{K,G}(s)\frac{X^s}{s}\right) = O(X^{\theta}),
        \]
        where the implied constant may depend on $K$ and $G\subseteq S_n$.
    \end{enumerate}
    If $s_1,s_2,...,s_r$ are ordered so that ${\rm Re}(s_1)\ge {\rm Re}(s_2)\ge \cdots \ge {\rm Re}(s_r)$ with decreasing order of the poles with identical real parts, we call any term $\underset{s=s_j}{\rm Res}\left(D_{K,G}(s)\frac{X^s}{s}\right)$ a \textbf{lower order term} if ${\rm Re}(s_j) < {\rm Re}(s_1)$ or ${\rm Re}(s_j)={\rm Re}(s_1)$ with $s_j$ having a lower order than $s_1$.
\end{definition}

A residue calculation implies that
\[
\underset{s=s_j}{\rm Res}\left(D_{K,G}(s)\frac{X^s}{s}\right) = X^{s_j} P_{K,G,s_j}(\log X),
\]
where $P_{K,G,s_j}$ is a polynomial of degree one less than the order of the pole at $s_j$. When $G$ is abelian, Theorem \ref{thm:generating_mero} tells us the locations (but not the orders) of all the poles of $D_{K,G}(s)$ with real part $>1/2a(G)$, so that
\[
\#\Surj(G_K,G;X) = \sum_{d\in \ind(G-\{1\})} X^{1/d} P_{K,G,1/d}(\log X) + O(X^{\max\{\theta(K,G),1/2a(G)\}+\epsilon})
\]
where $\deg P_{K,G,1/d} \le b_d(K,G)-1$. This gives an asymptotic expansion for $\#\Surj(G_K,G;X)$ that includes the possible existence of lower order terms of the same shape as the main term that was predicted by Malle.

\begin{remark}
    We note that Definition \ref{def:theta}(a) does not imply Definition \ref{def:theta}(b). While the meromorphic continuation allows one to apply a contour shifting argument to the line ${\rm Re}(s) = \theta+\epsilon$ on Perron's formula, this does not immediately correspond to a power saving bound on the error term. See Section \ref{sec:Tauberian} for the Tauberian theorem we will use to produce a power saving, and specifically Subsection \ref{subsec:common_error} for a discussion on why the largest region of meromorphic continuation is not the same as the power saving bound for the error.
\end{remark}

We then prove upper bounds for $\theta(K,G)$. When $G$ is abelian, we give an explicit upper bound in terms of subconvexity bounds for Hecke $L$-functions in Theorem \ref{thm:best_bound}, which we discuss in the main body of the paper. This is done so that the bounds we prove for $\theta(K,G)$ will improve as researchers prove new subconvexity bounds. By inputting the subconvexity bounds proven in \cite{sohne1997}, we prove the following unconditional result:

\begin{corollary}\label{cor:uncond}
    Let $K$ be a number field and $G$ an abelian group. Then for each $D\in \ind(G-\{1\})$ or $D=2a(G)$,
    \[
    \theta(K,G) \le \frac{1}{a(G)} - \frac{\displaystyle\frac{1}{a(G)} - \frac{1}{D}}{\displaystyle 1+\sum_{\substack{g\in G-\{1\}\\ \ind(g) < D}} \frac{[K:\Q]}{3}\left(1 - \frac{\ind(g)}{D}\right)}.
    \]
\end{corollary}

This implies that $\theta(K,G) < 1/a(G)$ really does give an unconditional power savings for abelian $G$, which was previously proven by Frei--Loughran--Newton \cite[Theorem 1.7]{frei-loughran-newton2018}. The best case scenario for subconvexity bounds is the generalized Lindel\"of hypothesis (which follows from GRH). We also prove the following conditional result:

\begin{corollary}\label{cor:GLH}
    Let $K$ be a number field and $G$ an abelian group. The generalized Lindel\"of hypothesis for Hecke $L$-functions implies that $\theta(K,G) \le \frac{1}{2a(G)}$.
\end{corollary}

Thus, we expect square root savings compared to the main term predicted by Malle. It is notable that $\ind(g) < 2a(G)$ for all $g\in G$ by definition, which means that the square root power savings in Corollary \ref{cor:GLH} reveals all of the (potential) lower order terms associated to poles at $s=1/\ind(g)$ given in Theorem \ref{thm:generating_mero} and Definition \ref{def:theta}.

Our methods give analogous results in certain cases when we count $G$-extensions that satisfy local conditions or when we order $G$-extensions by other invariants like the product of ramified primes. We have opted to discuss only the discriminant ordering in the main body of the paper for the sake of readability, but we include a discussion on how these results generalize in Appendix \ref{app:local_cond_and_ordering}.

\subsection{Lower Order Terms}

Definition \ref{def:theta} allows for the possibility that $P_{K,G,1/d}$ vanishes, corresponding to the possibility that $s=1/d$ is not a pole of the generating series. This is allowed by Theorem \ref{thm:generating_mero}, as this theorem only prescribes an upper bound for the order of the pole (and we consider a ``pole of order zero" to mean a removable singularity). It is known that the pole at $s=1/a(G)$ has order exactly $b_{a(G)}(K,G) = b(K,G)$, giving the main term predicted by Malle. The usual argument used in \cite{wright1989,wood2009,alberts-odorney2021} to compute the order of this pole is to bound the counting below by a similar count with restricted local conditions. However, this argument no longer works for lower order terms.

We discuss some of the obstacles to proving non-vanishing of lower order terms in Section \ref{sec:obstructions_lower_order}. At this time, we are not able to completely overcome these obstacles. However, in certain nice cases, we are able to use ad hoc arguments to prove the existence of some lower order terms.

\begin{theorem}\label{thm:nonvanishing}
    Let $G$ be an abelian group, $\ell$ the smallest prime dividing $|G|$, and $D_{\Q,G}(s)$ the generating series for $G$-extensions of $\Q$ ordered by discriminant. Define
    \[
    \bar{b}_d(\Q,G) = b_d(\Q,G) - \sum_{\substack{g\in \ind(G-\{1\})\\ \ind(g)<d}} [\Q(\zeta_{|\langle g\rangle|}):\Q]^{-1}\underset{s=\frac{\ind(g)}{d}}{\rm ord}\left(\zeta_{\Q(\zeta_{|\langle g\rangle|})}(s)\right)
    \]
    if this quantity is nonegative, and $\bar{b}_d(\Q,G) = 0$ otherwise.
    \begin{enumerate}[(i)]
        \item If $d=\frac{|G|(\ell-1)}{\ell}=a(G)$ is the minimal index, then $D_{\Q,G}(s)$ has a pole at $s=1/d$ of order $b_d(\Q,G)=\bar{b}_d(\Q,G)$.
        \item If $d\in \ind(G-\{1\})$ is such that
        \[
        \{g\in G-\{1\} : \ind(g) < d\} \subseteq \Phi(G)
        \]
        is contained in the Frattini subgroup, then $D_{\Q,G}(s)$ has a pole at $s=1/d$ of order $\bar{b}_d(\Q,G)$.
        \item If $G=C_2\times A$ for $|A|$ odd, and $d\in \ind(G-\{1\})$ is such that
        \[
        \{g\in G-\{1\} : \ind(g) < d\} \subseteq \Phi(G) \cup G[2]
        \]
        then $D_{\Q,G}(s)$ has a pole at $s=1/d$ of order $\bar{b}_d(\Q,G)$.
        \item If $G=C_n$ is cyclic and $d=n-1$ is the index of a generator, then $D_{\Q,G}(s)$ has a pole at $s=1/d$ of order $\bar{b}_d(\Q,G)$.
    \end{enumerate}
\end{theorem}

Theorem \ref{thm:nonvanishing}(i) is in fact proven by Wright \cite{wright1989} for all number fields $K$, as the pole at $s=1/a(G)$ corresponds to the main term of the asymptotic growth rate. We will give a new direct proof for this case when $K=\Q$, in line with the other cases. Theorem \ref{thm:nonvanishing} is proven by more ad hoc arguments, essentially by getting our hands dirty with alternating sums of convergent Euler products. We prove that these cases are ``nice enough" that we are able to treat or avoid the obstructions described in Section \ref{sec:examples_lower_order}.

Our discussion in Section \ref{sec:obstructions_lower_order} together with Theorem \ref{thm:nonvanishing} can be taken as evidence for the following conjecture extending Theorem \ref{thm:generating_mero}:

\begin{conjecture}\label{conj:Q_order_of_poles}
    Let $G$ be an abelian group and $D_{\Q,G}(s)$ the generating Dirichlet series for $G$-extensions of $\Q$ ordered by discriminant. Then $D_{\Q,G}(s)$ has a meromorphic continuation to the right halfplane ${\rm Re}(s) > \frac{1}{2a(G)}$ with poles at $s=\frac{1}{d}$ for each $d\in \ind(G-\{1\})$ of order \textbf{exactly}
    \[
        \max_{\Phi(G)\le H\le G}\bar{b}_{d/[G:H]}(\Q,H),
    \]
    where $\bar{b}_d(\Q,H)$ is defined as in Theorem \ref{thm:nonvanishing}.
\end{conjecture}

We remark that the index function of $H$ is related to the index function of $G$ by $\ind_H(h) = \frac{1}{[G:H]}\ind_G(h)$ for all $h\in H\le G$, which is where the $d/[G:H]$ factor comes from. Assuming $\zeta_{\Q(\zeta_m)}(s)$ has no rational roots of size $>1/2$, it follows the definitions that
\[
\max_{\Phi(G)\le H\le G}\bar{b}_{d/[G:H]}(\Q,H) = \max_{\Phi(G)\le H\le G}b_{d/[G:H]}(\Q,H) = b_d(\Q,G)
\]
agrees with the maximum value given in Theorem \ref{thm:generating_mero}. In principle, a number of individual Dedekind zeta function $\zeta_K(s)$ with $K$ of small discriminant are known to not have nontrivial rational zeros, but this is not known in general.

If there exists a rational counter example to the Riemann Hypothesis for $\Q(\zeta_m)$ for some $m$, it is possible for smaller $H\le G$ to determine this maximum. For an example of how this could happen, consider $C_{15}=\langle \sigma\rangle$ and $d=\ind(\sigma^3) = 12$. We can compute
\begin{align*}
    \bar{b}_{12}(\Q,C_{15}) &= \max\left\{1 - \underset{s=5/6}{\rm ord}\left(\zeta_{\Q(\zeta_{3})}(s)\right),0\right\}\\
    \bar{b}_{12/[C_{15}:C_5]}(\Q,C_{5}) &= 1.
\end{align*}
If $\zeta_{\Q(\zeta_3)}(s)$ has a zero at $s=5/6$, then $H=C_5$ gives the primary contribution to the pole at $s=1/12$. In the proof, this corresponds to a lower term in the sieve to $G$-extensions from $G$-\'etale algebras giving the order of the pole. This phenomenon occurs whenever $H$ contains all the elements in $G$ of index $d$, but does not contain all the elements of index smaller than $d$.

Theorem \ref{thm:nonvanishing}(ii,iii) can be used to deduce some cases of this conjecture. The following corollary is immediate from the structure of Theorem \ref{thm:nonvanishing} and the fact that $\zeta(s)$ has no nontrivial real zeros:

\begin{corollary}
    Conjecture \ref{conj:Q_order_of_poles} is true for $G=(C_{p^k})^n$ for any prime $p$ or $G=C_2\times (C_{p^k})^n$ for any prime $p>2$. Moreover, in each of these cases,
    \begin{align*}
        \max_{\Phi(G)\le H\le G}\bar{b}_{d/[G:H]}(\Q,H) = \bar{b}_d(\Q,G)
    \end{align*}
    for each $d \in \ind(G-\{1\})$.
\end{corollary}

A similar question can be asked for other number fields $K$, or in the twisted setting that we describe in the next subsection. However, these settings come with more delicate obstructions involving linear combinations of special values of $L$-functions. We do not currently have the evidence to justify extending the conjecture to such cases.

In order to turn the poles in Theorem \ref{thm:nonvanishing} into lower order terms, we need to compare these to the bounds for $\theta(K,G)$ in the error term. Corollary \ref{cor:GLH} immediately implies the following:
\begin{corollary}\label{cor:GRH_lower_order}
    Assume GRH. Then $\#\Surj(G_\Q,G;X)$ admits a lower order term if $G$ satisfies at least one of the following:
    \begin{enumerate}
        \item[(a)] $G$ is not elementary abelian, i.e. $\Phi(G) \ne 1$,
        \item[(b)] $G = C_2 \times A$ for a nontrivial abelian group $A$ of odd order, or
        \item[(c)] $G = C_n$ for some composite $n>1$.
    \end{enumerate}
    Moreover, if $G=(C_{p^k})^n$ for some prime $p$ or $G=C_2\times (C_{p^k})^n$ for some prime $p>2$ then $\#\Surj(G_{\Q},G;X)$ admits an asymptotic term of size
    \[
    X^{\frac{1}{\ind(g)}}P_{\Q,G,1/\ind(g)}(\log X)
    \]
    for each $g\in G-\{1\}$ for a polynomial $P_{\Q,G,1/\ind(g)}$ of degree exactly $b_{\ind(g)}(\Q,G)-1$.
\end{corollary}

\begin{remark}
    The lower order term referred to in Corollary \ref{cor:GRH_lower_order}(c) of order $X^{\frac{1}{n-1}}$ may not be the \emph{secondary term} specifically, while Corollary \ref{cor:GRH_lower_order}(a,b) does refer to a specific secondary (and possibly tertiary and so forth) term. This is because $n-1$ is the largest index of an element in $C_n$, as opposed to the ``next smallest". If $n$ is not the square of a prime, there are indices $\frac{1}{n-1} < \frac{1}{\ind(g)} < \frac{1}{a(C_n)}$ which are potentially associated to poles of $D_{\Q,C_n}(s)$ by Theorem \ref{thm:generating_mero}. However, stronger arguments are needed to prove that these poles do not cancel.
\end{remark}

Using the bounds in Corollary \ref{cor:uncond} instead is much more involved. Using MAGMA \cite{MAGMA}, we are able to collect same data for cyclic groups.

\begin{data}\label{data}
    We consider $\#\Surj(G_\Q,C_n;X)$ for composite $n < 20{,}000$.
    \begin{enumerate}[(i)]
        \item Corollary \ref{cor:uncond} reveals a potential lower order term for $\approx 48.5\%$ of composite $n<20{,}000$. That is, a lower order term is revealed if $s=1/d$ is a pole of $D_{\Q,C_n}(s)$ for $d$ the second smallest index of $C_n$.
        \item Corollary \ref{cor:uncond} and Theorem \ref{thm:nonvanishing} reveal a lower order for $\approx 39.4\%$ of composite $n<20{,}000$.
    \end{enumerate}
\end{data}

We prove this directly in some uncomplicated cases, primarily to show that $\#\Surj(G_\Q,G;X)$ admits an unconditional lower order term for infinitely many abelian groups $G$.
\begin{corollary}\label{cor:C4andC6}
    Let $M$ be a positive integer which is coprime to $6$. Then $\#\Surj(\Q,C_{4M};X)$ and $\#\Surj(\Q,C_{6M};X)$ admit an unconditional secondary term of size $X^{1/3M}$ and $X^{1/4M}$ respectively.
\end{corollary}
We remark that these account for $\approx 11.1\%$ of cyclic groups.

\subsection{Power Saving Bounds for Counting Twisted Abelian Extensions}
We also prove analogous power savings in the twisted setting developed by the author in \cite{alberts2021}. The twisted counting problem can be presented in two equivalent ways: Let $K$ be a number field, $G\subseteq S_n$ a transitive subgroup, and $T\normal G$ a normal subgroup.
\begin{enumerate}[(1)]
    \item Given any $G$-extension $L/K$, we can consider the counting function
    \[
    \#\{F/K : [F:K]=n,\ \Gal(\widetilde{F}/K)\cong G,\ \widetilde{F}^T = L^T,\ |\disc(F/K)|\le X\}.
    \]
    \item Alternatively, given a continuous homomorphism $\pi:G_K\to G$ we can consider the counting function
    \[
    \#\Surj(G_K,T,\pi;X) := \#\{f\in Z^1(K,T(\pi)) : f*\pi:G_K\to G\text{ is surjective},\ |\disc(f*\pi)|\le X\},
    \]
    where $T(\pi)$ is the group $T$ with the Galois action given by $x.t = \pi(x)t\pi(x)^{-1}$ and $(f*\pi)(x) = f(x)\pi(x)$.
\end{enumerate}
Up to a constant multiple, these two questions are equivalent if $L$ is the field fixed by $\ker \pi$. This is called a ``twisted" counting problem due to the nontrivial Galois action on $T$.

\begin{remark}
    The set $\Surj(G_K,T,\pi;X)$ depends on the choice of homomorphism $\pi:G_K\to G$, however the cardinality of this set $\#\Surj(G_K,T,\pi;X)$ depends only on the Galois action on $T(\pi)$. Different homomorphisms $\pi$ can produce the same Galois action, however the correspondence with the first perspective implies that these different choices do not change the size of the set. For this reason, we can write the asymptotic growth invariants as depending only on $K$ and $T(\pi)$.
\end{remark}

The author introduced this counting problem to formulate a natural generalization of Malle's conjecture to a twisted setting, which predicts that
\[
\#\Surj(G_K,T,\pi;X)\sim c(K,T(\pi)) X^{1/a(T)} (\log X)^{B(K,T(\pi)) - 1},
\]
where $c(K,T(\pi))>0$, $a(T) = \min_{t\in T-\{1\}} \ind(t)$, and $B(K,T(\pi)) \ge b(K,T(\pi))$ generalizes T\"urkelli's modified invariant. Notice that when $T=G$ this specializes to Malle's original conjecture (with T\"urkelli's modified invariant). For our purposes we will only consider the case that $T$ is abelian, for which
\[
B(K,T(\pi)) = b(K,T(\pi)) = \#\{\text{orbits }o\text{ of }T\text{ under }\pi*\chi^{-1}\text{ with }\ind(o)=a(T)\}
\]
where $\chi:G_K\to \hat{\Z}^{\times}$ is the cyclotomic character. The same asymptotic is expected to be true for
\[
\#H^1(K,T(\pi);X):=\frac{1}{|B^1(K,T(\pi))|}\#\{f\in Z^1(K,T(\pi)) : |\disc(f*\pi)|\le X\}
\]
as well, with a possibly different leading coefficient.

The author together with O'Dorney \cite{alberts-odorney2021} proved the twisted Malle's conjecture for any abelian $T$, giving the main asymptotic term for $\#\Surj(G_K,T,\pi;X)$. This proof uses Poisson summation associated to the Galois cohomology group $H^1(K,T(\pi))$ to express the generating Dirichlet series as a finite sum of Euler products. The expression of the generating series as a finite sum of Euler products is much like the argument used by Wright for counting abelian extensions \cite{wright1989}, where the use of Poisson summation in this context originates in \cite{frei-loughran-newton2018}. We leverage this decomposition together with a deeper study of the Tate pairing to prove the following generalization of Theorem \ref{thm:generating_mero}:

\begin{theorem}\label{thm:generating_mero_twist}
    Let $K$ be a number field, $G\subseteq S_n$ a transitive subgroup, $T\normal G$ abelian, and $\pi:G_K\to G$ a continuous homomorphism. The generating Dirichlet series for $f\in Z^1(K,T(\pi))$ for which $f*\pi$ is surjective ordered by discriminant is meromorphic on the right halfplane ${\rm Re}(s) > 0$.

    Any poles with real part larger than $\frac{1}{2a(T)}$ occur at $s=\frac{1}{d}$ for each $d\in \ind(T-\{1\})$  of order \textbf{at most} $b_d(K,T(\pi))$, where
    \[
    b_d(K,T(\pi)) = \#\{\text{orbits }o\text{ of }T\text{ under }\pi*\chi^{-1}\text{ such that }\ind(o)=d\}.
    \]
    Here, $\chi:G_K\to \hat{Z}^\times$ is the cyclotomic character, and $\pi*\chi^{-1}$ induces a Galois action on $T$ by $x.t = (\pi(x)t\pi(x)^{-1})^{\chi(g)^{-1}}$.
\end{theorem}

Once again, each pole is expected to contribute a term to the asymptotic growth rate. We can define the following invariant generalizing $\theta(K,G)$:

\begin{definition}\label{def:theta_twist}
    Let $K$ be a number field, $G\subseteq S_n$ a transitive group, $T\normal G$, $\pi:G_K\to G$ a continuous homomorphism, and $D_{K,T(\pi)}(s)$ the generating series for the set of crossed homomorphisms $f\in Z^1(K,T(\pi))$ for which $f*\pi$ is surjective ordered by discriminant. Define $\theta(K,T(\pi))$ to be the infimum of all real numbers $\theta$ for which
    \begin{enumerate}[(a)]
        \item $D_{K,T(\pi)}(s)$ has a meromorphic continuation to ${\rm Re}(s) > \theta$ with at most finitely many poles $s_1,s_2,...,s_r$, and
        \item The counting function satisfies the bound
        \[
            \#\Surj(G_K,T,\pi;X) - \sum_{j=1}^r \underset{s=s_j}{\rm Res}\left(D_{K,T(\pi)}(s)\frac{X^s}{s}\right) = O(X^{\theta}).
        \]
        where the implied constant may depend on $K$, $T\normal G\subseteq S_n$, and $\pi$.
    \end{enumerate}
    If $s_1,s_2,...,s_r$ are ordered so that ${\rm Re}(s_1)\ge {\rm Re}(s_2)\ge \cdots \ge {\rm Re}(s_r)$ with decreasing order of the poles with identical real parts, we call any term $\underset{s=s_j}{\rm Res}\left(D_{K,T(\pi)}(s)\frac{X^s}{s}\right)$ a \textbf{lower order term} if ${\rm Re}(s_j) < {\rm Re}(s_1)$ or ${\rm Re}(s_j)={\rm Re}(s_1)$ with $s_j$ having a lower order than $s_1$.
\end{definition}

The expression
\[
D_{K,T(\pi)}(s) = \sum_{\substack{\phi\in \Surj(G_K,G)\\\phi \equiv \pi \pmod T}} |\disc(\phi)|^{-s}
\]
shows that $D_{K,T(\pi)}(s)$ depends only on $\pi \pmod T$, i.e. depends only on the Galois action $T(\pi)$ and not on $\pi$.

We also give an explicit upper bound for $\theta(K,T(\pi))$ when $T$ is abelian in terms of subconvexity bounds for Hecke $L$-functions in Theorem \ref{thm:best_bound}. The subconvexity bounds proven in \cite{sohne1997} give the following unconditional result:

\begin{corollary}\label{cor:uncond_twist}
    Let $K$ be a number field, $G\subseteq S_n$ a transitive subgroup, $T\normal G$ abelian, and $\pi:G_K\to G$ a continuous homomorphism. Then for each $D\in \ind(T-\{1\})$ or $D=2a(T)$,
    \[
    \theta(K,T(\pi)) \le \frac{1}{a(T)} - \frac{\displaystyle\frac{1}{a(T)} - \frac{1}{D}}{\displaystyle 1+\sum_{\substack{g\in T-\{1\}\\\ind(g) < D}} \frac{[K:\Q]}{3}\left(1 - \frac{\ind(g)}{D}\right)}.
    \]
\end{corollary}

This proves that $\theta(K,T(\pi)) < 1/a(G)$ gives an unconditional power savings for abelian $T$. In fact, the bound in Corollary \ref{cor:uncond_twist} is the same as for Corollary \ref{cor:uncond} where at most the index function changes depending on the transitive representation of $G$. In particular, for fixed transitive $G\subseteq S_n$ and fixed abelian $T\normal G$, the unconditional bound on $\theta(K,T(\pi))$ is independent of $\pi$, i.e. independent of the Galois action on $T$ induced by $\pi$.

Finally, we again have the best case scenario that follows from the generalized Lindel\"of hypothesis:

\begin{corollary}\label{cor:GLH_twist}
    Let $K$ be a number field, $G\subseteq S_n$ a transitive subgroup, $T\normal G$ abelian, and $\pi:G_K\to G$ a continuous homomorphism. The generalized Lindel\"of hypothesis for Hecke $L$-functions implies that $\theta(K,T(\pi)) \le \frac{1}{2a(T)}$.
\end{corollary}

Thus, we expect square root savings compared to the main term predicted in \cite{alberts2021}. However, in the twisted setting some (potential) lower order terms appearing in Definition \ref{def:theta_twist} coming from poles in Theorem \ref{thm:generating_mero_twist} may be encompassed by this error term. For example, consider $T=C_2\times C_2\normal D_4=C_2\wr C_2$. In this case $a(T) = \ind(g,1)=\ind(1,g)=1$, but $\ind(g,g) = 2 = 2a(T)$.

\subsection{History of Previous Results}\label{subsec:history}

Malle's conjecture is known to hold in numerous cases, including
\begin{itemize}
    \item abelian groups \cite{wright1989},
    \item $S_n$ in degree $n$ for $n=3,4,5$ \cite{bhargava2014,bhargava-shankar-wang2015},
    \item $S_n\times A$ for $n=3,4,5$ and $A$ any abelian group \cite{jwang2021},
    \item $D_4$ in degree $4$ \cite{cohen-diaz-y-diaz-olivier2005},
    \item generalized quaternion groups \cite{klunersHab2005},
    \item nilpotent groups for which $\{g\in G-\{1\} : \ind(g) = a(G)\}$ is central \cite{koymans-pagano2021}, and
    \item \emph{most} wreath products $C_2\wr H$ \cite{kluners2012}.
\end{itemize}
In an upcoming joint work of the author with Lemke Oliver, Wang, and Wood, we will prove Malle's conjecture for a much larger family of groups depending on the locations of minimal index elements in the group.

The methods in each of these cases will actually prove
\[
\#\Surj(G_K,G;X) = X^{1/a(G)}P(\log X) + o(X^{1/a(G)}),
\]
where $P$ is a polynomial of degree $b(K,G)$ with positive leading coefficient. This quality of error term is not always mentioned explicitly, but often readily follows from the author's methods (most often a Tauberian theorem).

Power saving error bounds are considerably rarer. The author is aware of only the following three cases for which a power saving error bound is known. In all three of these cases, the original power savings is achieved by a shifting contour argument.
\begin{itemize}
    \item There is a power saving error for counting $S_3$-extensions of $\Q$, revealing a secondary term
    \[
    \#\Surj(G_{\Q},S_3;X) = c_1 X + c_2 X^{5/6} + O_{\epsilon}(X^{2/3+\epsilon})
    \]
    due to Bhargava--Taniguchi--Thorne \cite{bhargava-taniguchi-thorne2023} for explicit constants $c_1 > 0$ and $c_2 < 0$, giving improved power savings over \cite{bhargava-shankar-tsimerman2012,taniguchi-thorne2013}. This has been lifted to a power saving error bound revealing an analogous secondary term for $S_3\times A$-extensions of $\Q$ for abelian groups $A$ by Wang \cite{jwang2017} using inductive methods. There do exist similar power saving error bounds over other number fields, but they are not strong enough to reveal the conjectured secondary term at $X^{5/6}$.
    \item Two recent works have proven power saving error bounds for counting $D_4$-extensions, namely
    \[
    \#\Surj(G_\Q,D_4;X) = c(\Q,D_4)X + O(X^{5/8+\epsilon})
    \]
    for an explicit constant $c(\Q,D_4) > 0$ due to \cite{mcgown-tucker2023}, improving an error term of size $O(X^{3/4+\epsilon})$ proven by \cite{cohen-diaz-y-diaz-olivier2002}, and
    \[
    \#\Surj(G_K,D_4;X) = c(K,D_4)X + O(X^{3/4+\epsilon})
    \]
    for an explicit constant $c(K,D_4) > 0$ due to \cite{bucur-florea-serrano-lopez-varma2022}, which improves to $O(X^{1/2+\epsilon})$ assuming the Lindel\"of hypothesis.
    \item Abelian groups are known to have power saving error terms. Frei--Loughran--Newton \cite[Theorem 1.7]{frei-loughran-newton2018} prove the existence of a power saving error term for counting abelian extensions ordered by discriminant with restricted local conditions, although they do not give an explicit value for the power saving exponent. In essence, they prove that $\theta(K,G) < 1/a(G)$ for all number fields $K$ and abelian $G$.
    
    Explicit power saving error bounds are known for small abelian groups, although these are seldom mentioned in the literature. When compiling data on the distributions of extensions with degree $\le 4$, Cohen--Diaz y Diaz--Olivier show in \cite{cohen-diaz-y-diaz-olivier2006} that the data supports a secondary term with power saving error bound for $C_4$-extensions of $\Q$ of the form
    \[
    \#\Surj(G_\Q,C_4;X) = c_1 X^{1/2} + c_2 X^{1/3} + o(X^{1/3}).
    \]
    They remark that this follows from standard contour integration techniques, despite being unable to find such a result in the literature.

    Some small abelian groups correspond to well studied sequences of integers for which power savings are known, for instance power saving error bounds for counting squarefree numbers translate to power saving error bounds for counting quadratic extensions. The ``standard contour integration" techniques referred to by Cohen--Diaz y Diaz--Olivier are equally applicable to other small abelian groups whose generating series have a known meromorphic continuation, but this treatment does not appear in full generality in the literature.

    Lee--Oh \cite{lee-oh2012} study $\theta(\Q,C_p)$ for cyclic groups $C_p$ of prime order $p\ge 3$. They prove an analytic continuation for the generating series of $C_p$-extensions over $\Q$ to the right halfplane ${\rm Re}(s) > \frac{1}{4a(C_p)}$ with explicit locations of the poles. Their proof is a special case of Theorem \ref{thm:generating_mero}, where they use more details of Moroz's technique \cite{moroz_1988} to gives descriptions of the poles in the strip $\frac{1}{4a(C_p)}< {\rm Re}(s) \le \frac{1}{2a(C_p)}$. However, we believe their work has an error when applying a Tauberian theorem to bound $\theta(\Q,C_p)$. They appear to conflate the region of meromorphic continuation with the power savings for the error term. We give an explanation for the suspected gap in \cite{lee-oh2012} in Subsection \ref{subsec:common_error}.
\end{itemize}

Our main results realize Cohen--Diaz y Diaz--Olivier's observation for $C_4$, and rigorously generalize it to all abelian groups. Despite Cohen--Diaz y Diaz--Olivier's observation, this is not immediate in general.

For another small example, consider the generating series of $C_3$ extensions of $\Q$. This is given by
\[
1 + \sum_{\substack{K/\Q\\\Gal(K/\Q)\cong C_3}} \disc(K/\Q)^{-s} = \frac{1}{2}\left(1 + 2\cdot 3^{-4s}\right)\prod_{p\equiv 1(\text{mod}\ 3)}\left(1+2p^{-2s}\right).
\]
The standard process reduces this to $\zeta(2s)L(2s,\chi_3)G(s)$ where $\chi_3$ is the quadratic character associated to $\Q(\zeta_3)/\Q$ and $G(s)$ is an absolutely convergent Dirichlet series on the region ${\rm Re}(s)>1/4$. By applying a Tauberian theorem and shifting the contour integral to the line ${\rm Re}(s) = 1/4+\epsilon$, we produce a power savings of the form
\[
c_1 X^{1/2} + O_{\epsilon}(X^{5/16+\epsilon}).
\]
using S\"ohne's subconvexity bounds on Dirichlet $L$-functions (This is the power savings given by Corollary \ref{cor:uncond}). Under GRH, this can be improved to $O(X^{1/4+\epsilon})$ using the same line of integration.

For general $K$ and abelian $G$, there may be more than one Euler product with more complicated Euler factors. We will similarly prove a meromorphic continuation by factoring out powers of $L$-functions. The primary obstacle to generalizing this example is proving that the meromorphic continuation is achieved by factoring out specifically \emph{positive integer} powers of $L$-functions, so that no branch cuts are introduced. Proving that this is the case, which is done by proving new extensions of the local Tate pairing, can be considered the main insight of this paper.

\subsection{Layout of the Paper}

This paper has three components:
\begin{enumerate}[(1)]
    \item an algebraic component, which proves the necessary facts to conclude the meromorphic continuations in Theorem \ref{thm:generating_mero} and \ref{thm:generating_mero_twist}. This consists of Sections \ref{sec:prelim}, \ref{sec:mero}, \ref{sec:local_Tate}, and \ref{sec:Galois_rep}.
    
    \item an analytic component, which performs the contour shifting argument to prove Theorem \ref{thm:xi+1} in Section \ref{sec:Tauberian}, which is a new Tauberian theorem. We then proceed to analyze the bounds produced by Theorem \ref{thm:xi+1} to prove the remaining power saving results in the Introduction in Sections \ref{sec:subconvexity} and \ref{sec:optimizing}.
    
    \item an investigation of lower order terms. In Section \ref{sec:obstructions_lower_order} we describe the obstructions to proving that the poles given by Theorems \ref{thm:generating_mero} and \ref{thm:generating_mero_twist} do not vanish. We then prove non-vanishing of these poles in some very simple cases in Section \ref{sec:examples_lower_order}.
\end{enumerate}

The algebraic section begins with Section \ref{sec:prelim} on preliminaries. We summarize the main results of \cite{alberts-odorney2021} on the formal decomposition of generating Dirichlet series and prove some basic representation theory facts that will be needed later in the paper.

In Section \ref{sec:mero} we state Theorem \ref{thm:euler_factors}, the main algebraic result of the paper which gives an explicit description of the Euler factors appears in the generating Dirichlet series. These factors are all of the form
\[
1 + \sum_{j} \chi_j(\Fr_p) p^{-a_j s}
\]
for some positive integers $a_j$ and Galois characters $\chi_j$. It was proven by Moroz \cite{moroz_1988} that such Euler products have a meromorphic continuation to ${\rm Re}(s) > 0$. By doing some of the steps in Moroz's proof explicitly, we show that Theorems \ref{thm:generating_mero} and \ref{thm:generating_mero_twist} follow from this description.

We then prove Theorem \ref{thm:euler_factors} in Sections \ref{sec:local_Tate} and \ref{sec:Galois_rep}. Section \ref{sec:local_Tate} contains the Galois cohomology needed to describe the Euler factors. This is the deepest section of the paper: the Euler factors are described in \cite{alberts-odorney2021} using the local Tate pairing. We utilize the theory of cup products and the Hochschild-Serre spectral sequence in the process of unpacking the Euler factors. Section \ref{sec:Galois_rep} contains the representation theory needed to simplify the Euler factors into factors involving Galois characters.

The analytic portion of the paper beings with Section \ref{sec:Tauberian}, which states and proves the Tauberian theorem we will use to convert the meromorphic continuation into asymptotic information. This is Theorem \ref{thm:xi+1}. The proof is largely similar to the proof of Landau's Tauberian theorem found (in French) in \cite{roux2011}, with Theorem \ref{thm:xi+1} slightly improving the power savings under the same hypotheses. We present the proof in its entirety to showcase the adjustments needed to improve the power savings and for the sake of completeness of this paper.

Section \ref{sec:subconvexity} is dedicated to discussing the topic of subconvexity. We describe the problem in general along with some of the standard facts, but we do not attempt to provide a complete account of the best known subconvexity bounds. We then formulate the power savings in terms of subconvexity bounds in Section \ref{sec:optimizing}, which requires some optimization. Essentially, the optimal choice for the line that we shift the contour to in Theorem \ref{thm:xi+1} depends on the strength of the subconvexity bounds and we give a full account of this phenomenon.

We then discuss obstructions to the existence of lower order terms in Section \ref{sec:obstructions_lower_order}, and use some ad hoc arguments to prove that the lower order terms do not vanish in certain cases via Theorem \ref{thm:nonvanishing} in Section \ref{sec:examples_lower_order}.

Lastly, we conclude with Appendix \ref{app:local_cond_and_ordering} on counting $G$-extensions satisfying restricted local conditions or ordered by other invariants. Our methods extend naturally to these questions, and under certain natural conditions on the families of local conditions and invariants our methods prove analogous results. We summarize how our method applies to these generalizations in the appendix.

\section*{Acknowledgements} The author would like to thank Robert Lemke Oliver for numerous conversations on the analytic details of this paper. The author would also like to thank Alina Bucur and Kiran Kedlaya for helpful conversations, as well as Daniel Loughran, Gunter Malle, and Evan O'Dorney for valuable feedback.

\section{Preliminaries}\label{sec:prelim}

\subsection{Summary of Alberts--O'Dorney}\label{subsec:AO}

The author in a joint paper with O'Dorney \cite{alberts-odorney2021} gave the asymptotic main term of the counting functions
\[
H^1(K,T(\pi);X) = \{f\in H^1(K,T(\pi)):|\disc(f*\pi)|\le X\}
\]
by giving a decomposition for the generating series
\[
\sum_{f\in H^1(K,T(\pi))} |\disc(f*\pi)|^{-s}.
\]
This decomposition is produced in a natural way via Poisson summation. We summarize the main points here: $H^1(K,T(\pi))$ is (up to a finite index subgroup) a discrete subgroup of the adelic product
\[
H^1(\mathbb{A}_K,T(\pi)) := \prod_p{}' H^1(K_p,T(\pi))_{H^1_{ur}(K_p,T(\pi))}.
\]
By local Tate duality and the Poitou-Tate nine term exact sequence, $H^1(K,T(\pi))$ is exactly annihilated by the image of $H^1(K,T(\pi)^*)$ in $H^1(\mathbb{A}_K,T(\pi)^*)$, where $T(\pi)^* = \Hom(T(\pi),\mu)$ is the Tate dual. Thus, setting $w_s(f) = |\disc(f*\pi)|^{-s}$ for a fixed complex number $s$ with ${\rm Re}(s)>1$, Poisson summation via \cite[Theorem 2.3]{alberts-odorney2021} implies
\[
\sum_{f\in H^1(K,T(\pi))} |\disc(f*\pi)|^{-s} = \frac{|H^0(K,T(\pi))|}{|H^0(K,T(\pi)^*)|}\sum_{h\in H^1(K,T(\pi)^*)} \widehat{w_s}(h),
\]
Well known properties of $\disc(f*\pi)$ imply very nice properties for $\widehat{w_s}$. We state these, but leave the ideas behind their proof to \cite{alberts-odorney2021}:
\begin{itemize}
    \item $w_s$ is periodic with respect to $\prod_{p\mid \infty} H^1(G_{K_p},T(\pi))\times\prod_{p \nmid \infty} H^1_{ur}(K_p,T(\pi))$, as we know the discriminant depends only on ramification. This implies $\widehat{w_s}$ is supported on the corresponding annihilator $\prod_{p\mid \infty} H^1_{ur}(K_p,T(\pi))^{\perp}$. In particular, the righthand summation is supported on the \emph{finite} set
    \[
    H^1_{ur^{\perp}}(K,T(\pi)^*) := H^1(K,T(\pi)^*) \cap \prod_{p\mid \infty} H^1_{ur}(K_p,T(\pi))^{\perp}.
    \]
    Finiteness follows from the fact that for all places $p\nmid |T|\disc(\pi)\infty$ it follows that $H^1_{ur}(K_p,T(\pi))^{\perp} = H^1_{ur}(K_p,T(\pi)^*)$. This is the dual Selmer group to $H^1_{ur}(G_{\Q},T(\pi))$.
    \item $w_s$ is multiplicative, which implies $\widehat{w_s}$ is also multiplicative. In particular, $\widehat{w_s}(h)$ is given by the Euler product
    \[
    \widehat{w_s}(h) = \prod_p \left(\frac{1}{|H^0(K_p,T(\pi))|}\sum_{f\in H^1(K_p,T(\pi))} \langle f,h_p\rangle |\disc(f*\pi)|^{-s}\right),
    \]
    where $h_p = \res_{G_{K_p}} h$ is the restriction of $h$ to the decomposition group at $p$ and $\langle,\rangle$ is the local Tate pairing.
\end{itemize}
The asymptotic growth rate is computed in the usual way by applying a Tauberian theorem to each of the finitely many Euler products appearing in the sum, with the rightmost pole of $\widehat{w_s}(0)$ supplying the main term. It is shown that the rightmost pole of $\widehat{w_s}(0)$ is not cancelled out by the poles of $\widehat{w_s}(h)$, nor is it cancelled out by the sieve to surjective coclasses $\Surj(G_K,T,\pi;X)$, by which it is concluded that the main term has a positive leading coefficient.\footnote{When counting coclasses with restricted local conditions, it is uncommon but possible for the main term to cancel out. This occurs when the restricted local conditions form a Grunwald--Wang counter example, and in these cases everything cancels out and there are precisely \emph{zero} coclasses satisfying those particular local conditions. This possibility was overlooked in the original results in \cite{alberts-odorney2021}, and is addressed in a corrigendum \cite{alberts-odorney2023}. We will not consider restricted local conditions in this paper, so we will not need to watch out for Grunwald--Wang counter examples.}

A key step in applying the Tauberian theorem in \cite{alberts-odorney2021} is constructing a meromorphic continuation of each Euler product $\widehat{w_s}(h)$. This is done via \cite[Proposition 2.2]{alberts2021}, as the map $p\mapsto \langle f, h_p\rangle$ given by the local Tate pairing is a Frobenian map, i.e. all but finitely many factors are determined by the splitting type of $p$ in some finite extension. This is essential for proving the asymptotic results of \cite{alberts2021,alberts-odorney2021}.

We will give an alternate construction for the meromorphic continuation, which is more explicit and significantly stronger than the results used in \cite{alberts2021,alberts-odorney2021}. This improvement will directly result in power savings after applying a Tauberian theorem.

\subsection{Twisted Permutation Representation}

As part of our work producing meromorphic continuations, we will make use of Artin $L$-functions associated to a special type of representation.

\begin{definition}
A representation $\rho:G\to {\rm GL}_n(\C)$ is called a \textbf{twisted permutation representation} if, for each $g\in G$, the matrix $\rho(g)$ has exactly one nonzero entry in each row and column. These have also been called \textbf{monomial representations} in the literature.

In this case, call $\overline{\rho}:G\to {\rm GL}_n(\C)$ the associated permutation representation defined by letting $\overline{\rho}(g)$ have a $1$ in every entry that $\rho(g)$ is nonzero, and a $0$ in every other entry.
\end{definition}

This is a generalization of a permutation representation, for which every row and column has exactly one nonzero entry given by $1$. If $\rho:G_K \to {\rm GL}_n(\C)$ is a permutation representation, then the Artin $L$-function $L(s,\rho)$ is analytically continued by way of the decomposition
\[
\rho = \bigoplus_o \ind_{\Stab(o)}^{G_K}(1)
\]
where the sum is over the orbits of the permutation action induced by $\rho$ on the standard basis vectors of $\C^n$. Recalling that Artin $L$-functions are invariant under induction, this implies
\[
L(s,\rho) = \prod_{o} \zeta_{\overline{K}^{\Stab(o)}}(s),
\]
which is meromorphic with a single pole at $s=1$ with order given by the number of orbits. This process generalizes for twisted permutation representations:

\begin{lemma}\label{lem:twistedperm}
If $\rho:G\to {\rm GL}_n(\C)$ is a twisted permutation representation for which $\overline{\rho}$ is a transitive permutation representation, then
\[
\rho = \ind_{\Stab(1)}^G(\rho|_{\Stab(1)} | \langle \vec{e}_1\rangle),
\]
where $\vec{e}_1$ is the basis vector associated to $1\in \{1,...,n\}$ acted on by $G$ via $\overline{\rho}$.
\end{lemma}

This implies that the Artin $L$-fucntion associated to a twisted permutation representation $\rho:G_K\to {\rm GL}_n(\C)$ decomposes as
\[
L(s,\rho) = \prod_{o} L(s, \rho|_{\Stab(i_o)} | \langle \vec{e}_{i_o}\rangle),
\]
where $i_o$ is a choice of base point for the orbit. In particular, these are one dimensional representations over $\Stab(i_o)$ and therefore given by Hecke $L$-functions. Hecke $L$-functions are entire (except for the trivial character, in which case they have a simple pole at $s=1$). This immediately proves the following:

\begin{corollary}\label{cor:twisted_perm_mero}
    Let $\rho:G_K\to {\rm GL}_n(\C)$ be a twisted permutation representation. Then $L(s,\rho)$ is meromorphic with a pole at $s=1$ whose order is given by the number of orbits $o\subseteq \{1,...,n\}$ under the action given by $\overline{\rho}$ such that each basis vector $\vec{e}_i$ with $i\in o$ is an eigenvalue of $\rho(g)$ with eigenvalue $1$ for all $g\in \Stab(i)$.
\end{corollary}

\begin{proof}[Proof of Lemma \ref{lem:twistedperm}]
Write $V$ for the $G$-module induced by $\rho$. Decompose $V = \bigoplus V_i$ as a direct sum of irreducible $\Stab(1)$-modules, one of which is necessarily $V_1 = \langle \vec{e}_1\rangle$. It then follows that the function
\[
\C[G]\otimes_{\C[\Stab(1)]} V_1 \to V
\]
defined by
\[
\left(\sum_g a_g g\right) \otimes v \mapsto \sum_g a_g g.v
\]
is an isomorphism of $G$-modules.
\end{proof}

\section{Explicit Meromorphic Continuation}\label{sec:mero}

Fix some $h\in H^1(K,T(\pi)^*)$ and a map $w_s:H^1(\A_K,T(\pi))\to \C$ given by $w_s(f) = |\disc(f*\pi)|^{-s}$ for each complex number $s\in C$. Our goal is to better understand the Fourier transform
\[
\widehat{w_s}(h) = \prod_p \left(\frac{1}{|H^0(K_p,T(\pi))|}\sum_{f_p\in H^1(K_p,T(\pi))}\langle f_p, h|_{G_{K_p}}\rangle \disc(f_p*\pi|_{G_{K_p}})^{-s}\right),
\]
so that we can give a stronger, more explicit meromorphic continuation than the one used in \cite{alberts-odorney2021}. This is achieved by the following theorem:

\begin{theorem}\label{thm:euler_factors}
    Fix some $h\in H^1(K,T(\pi)^*)$. Then there exists a twisted permutation representation
    \[
    \rho_h:G_K \to {\rm GL}(\C[\Hom((\Q/\Z)^*,T(\pi))])
    \]
    (given in Definition \ref{def:twisted_perm_rep}), where $(\Q/\Z)^* = \Hom(\Q/\Z,\mu)$ is the Tate dual of $\Q/\Z$, such that for each place $p\nmid |T|\disc(\pi)\disc(h)\infty$
    \[
    \frac{1}{|H^0(K_p,T(\pi))|}\sum_{f_p\in H^1(K_p,T(\pi))}\langle f_p, h|_{G_{K_p}}\rangle \disc(f_p*\pi|_{G_{K_p}})^{-s} = \sum_{o\subset \Hom((\Q/\Z)^*,T(\pi))} {\rm tr}(\rho_h|_{\C o}(\Fr_p)) p^{-\ind(o)s},
    \]
    where the sum is over the orbits $o$ of $\Hom((\Q/\Z)^*,T(\pi))$ under the Galois action. Notice that $\ind(1) = 0$ so that $p^{-\ind(1)s} = 1$ gives the constant term.
\end{theorem}

The group $(\Q/\Z)^*$ is a procyclic group with Galois action given by the cyclotomic character $\chi:G_K\to \hat{\Z}^{\times}$. Choosing a generator gives a noncanonical isomorphism of groups $(\Q/\Z)^*\cong \hat{\Z}$, and therefore a noncanonical isomorphism of groups
\[
\Hom((\Q/\Z)^*,T(\pi)) \cong T.
\]
The index of an orbit $o$ is defined to be the index of the corresponding element of $T$. This is independent of the choice of generator of $(\Q/\Z)^*$, as $\ind(t)$ depends only on the subgroup $\langle t\rangle$ (or in other words, it is invariant under invertible powers).

The orbits of $\Hom((\Q/\Z)^*,T(\pi))$ can be considered as tame ramification types. Indeed, $I_p^{\rm tame}$ is canonically isomorphic to the prime-to-$p$ part of $(\Q/\Z)^*$. By writing all but finitely many Euler factors as a sum of characters, we can give a very explicit meromorphic continuation for $\widehat{w_s}(h)$. These types of Euler products were studied by Moroz \cite{moroz_1988}, who proved that they all have a meromorphic continuation to the region ${\rm Re}(s) > 0$.

\begin{corollary}\label{cor:mero_moroz}
    Let $w_s:H^1(\A_K,T(\pi))\to \C$ be given by $w_s(f) = |\disc(f*\pi)|^{-s}$ for each complex $s\in \C$ and fix some $h\in H^1(K,T(\pi)^*)$. Then $\widehat{w_s}(h)$ has a meromorphic continuation to the region ${\rm Re}(s) > 0$.
\end{corollary}

We will need specific information about the meromorphic continuation, including the locations of poles and bounds in vertical strips, in order to use a Tauberian theorem. At the same time, we will not need the full region of meromorphic continuation proven by Moroz. Due to growth in vertical strips, it is not optimal to shift the contour integral used in a Tauberian theorem all the way to ${\rm Re}(s) = \epsilon$ for some arbitrarily small $\epsilon > 0$. It will suffice for us to use the following, more explicit result:

\begin{corollary}\label{cor:meromorphic_continuation}
    Let $w_s:H^1(\A_K,T(\pi))\to \C$ be given by $w_s(f) = |\disc(f*\pi)|^{-s}$ for each complex $s\in \C$, $h\in H^1(K,T(\pi)^*)$, and $\rho_h$ the representation given by Definition \ref{def:twisted_perm_rep}. Then
    \[
    \widehat{w_s}(h) = B(s) \prod_{o\subseteq \Hom((\Q/\Z)^*,T(\pi))} L(\ind(o)s, \rho_h|_{\C o}),
    \]
    where $B(s)$ is an Euler product that converges absolutely in the region ${\rm Re}(s) > \frac{1}{2a(T)}$.
\end{corollary}

\begin{proof}
    One checks that
    \[
    1 + b_0x^a + b_1x^{a+1} + \cdots + b_Nx^{a+N} = (1+b_0x^{a})(1+b_1x^{a+1})\cdots(1+b_Nx^{a+N})(1 + O(x^{2a}))
    \]
    via the distributive law. This implies that
    \begin{align*}
    &\prod_{p\nmid |T|\disc(\pi)\disc(h)\infty} \sum_{o\subset \Hom((\Q/\Z)^*,T(\pi))} {\rm tr}(\rho_h|_{\C o}(\Fr_p)) p^{-\ind(o)s}\\
    &= \prod_{p\nmid |T|\disc(\pi)\disc(h)\infty}(1+O(p^{-2a(T)s}))\prod_{o\subset \Hom((\Q/\Z)^*,T(\pi))}\prod_{p\nmid |T|\disc(\pi)\infty}\left(1 + {\rm tr}(\rho_h|_{\C o}(\Fr_p)) p^{-\ind(o)s}\right),
    \end{align*}
    as $a(T) = \min \ind(o)$ over all orbits $o\ne \{1\}$. We also know that
    \[
    \left(1 + {\rm tr}(\rho_h|_{\C o}(\Fr_p)) p^{-\ind(o)s}\right) = \det\left(I + \rho_h|_{\C o}(\Fr_p)p^{-\ind(o)s}\right) (1 + O(p^{-2\ind(o)s})),
    \]
    so that we conclude
    \begin{align*}
        &\prod_{p\nmid |T|\disc(\pi)\disc(h)\infty} \sum_{o\subset \Hom((\Q/\Z)^*,T(\pi))} {\rm tr}(\rho_h|_{\C o}(\Fr_p)) p^{-\ind(o)s}\\
        &= \prod_{p\nmid |T|\disc(\pi)\disc(h)\infty}(1+O(p^{-2a(T)s}))\prod_{o\subset \Hom((\Q/\Z)^*,T(\pi))}\prod_{p\nmid |T|\disc(\pi)\infty}\det\left(I - \rho_h|_{\C o}(\Fr_p)p^{-\ind(o)s}\right)^{-1}.
    \end{align*}
    There are only finitely many primes $p\mid |T|\disc(\pi)\disc(h)\infty$. By including the Euler factors at these places for $\widehat{w_s}(h)$ on the left and for the Artin $L$-functions on the right, we have proven that
    \[
    \widehat{w_s}(h) = B(s)\prod_{o\subseteq \Hom((\Q/\Z)^*,T(\pi))} L(\ind(o)s, \rho_h|_{\C o})
    \]
    where $B(s)$ is an Euler product whose Euler factors are $1+O(p^{-2a(T)s})$ at all but finitely many places. At the other places, the Euler factors of $B(s)$ are given by
    \begin{align*}
    &\left(\frac{1}{|H^0(K_p,T)|}\sum_{f_p\in H^1(K_p,T)}\langle f_p,h|_{G_{K_p}}\rangle_p |\disc(f_p*\pi|_{G_{K_p}})|^{-s}\right)\\
    &\times\prod_{o\subset \Hom((\Q/\Z)^*,T(\pi))}\det\left(I - \rho_h|_{(\C o)^{I_p}}(\Fr_p)p^{-\ind(o)s}\right).
    \end{align*}
    Each Euler factor of $B(s)$ is an entire function, and the fact that all but finitely many of the Euler factors are $1+O(p^{-2a(T)s})$ implies that $B(s)$ converges absolutely on the right halfplane ${\rm Re}(s) > \frac{1}{2a(T)}$.
\end{proof}

Definition \ref{def:twisted_perm_rep} implies that $\rho_h$ is a twisted permutation representation, and therefore each of $\rho_h|_{\C o}$ is induced from a Hecke character. Artin $L$-functions are invariant under inducing the representation, and Hecke $L$-functions are known to satisfy the Artin conjecture. This implies the following meromorphic continuation:

\begin{corollary}\label{cor:poles}
    Let $w_s:H^1(\A_K,T(\pi))\to \C$ be given by $w_s(f) = |\disc(f*\pi)|^{-s}$ for each complex $s\in \C$ and $h\in H^1(K,T(\pi)^*)$. The Dirichlet series $\widehat{w_s}(h)$ has a meromorphic continuation to ${\rm Re}(s) > \frac{1}{2a(T)}$ with poles at $s=1/n$ for each $n\in \ind(T\setminus\{1\})$ of order
    \begin{align*}
    \le &\#\left\{o\subseteq \Hom((\Q/\Z)^*,T(\pi)) : \substack{\displaystyle \rho_h|_{\C o} = \ind_{L}^K(1_L)\text{ for }L\text{ the fixed field}\\ \displaystyle\text{of }\Stab_{G_K}(e)\text{ for some }e\in o}\right\}\\
    &- \sum_{\substack{o\subseteq \Hom((\Q/\Z)^*,T(\pi))\\ \ind(o) < n}} \underset{s=1/n}{\rm ord} L(\ind(o)s, \rho_h|_{\C o}),
    \end{align*}
    with equality if $h=0$. Here, $\underset{s=a}{\rm ord}f(s)$ denotes the order of the zero of $f(s)$ at $s=a$ and $1_L$ denotes the trivial representation of $G_L$. In particular, $\rho_h|_{\C o} = \ind_{L}^K(1_L)$ if and only if $\rho_h|_{\C o}$ is a permutation representation.
\end{corollary}

We note that it is conjectured that Artin $L$-functions satisfy GRH, and therefore have no zeros with real part larger than 1/2. Under GRH, the second term in Corollary \ref{cor:poles} would vanish. However, this is unproven so we need to be careful to include the possibility that other real zeros exist. In particular, it is possible for a Siegel zero to affect the order of poles to the left of $s=1/a(T)$.

Theorems \ref{thm:generating_mero} and \ref{thm:generating_mero_twist} follow immediately from Corollaries \ref{cor:mero_moroz} and Corollary \ref{cor:poles} together with the sieve to surjective coclasses. The most that the sieve to surjective coclasses can do is (partially) cancel out poles, so the upper bound on the orders of the poles is invariant under the sieve.

\begin{proof}[{Proof of Corollary \ref{cor:poles}}]
    The region of meromorphicity is immediate from Corollary \ref{cor:meromorphic_continuation}. The poles all come from the Artin $L$-functions, and the order of the pole $s=1/n$ in the product
    \[
    \prod_{o\subseteq \Hom((\Q/\Z)^*,T(\pi))} L(\ind(o)s, \rho_h|_{\C o})
    \]
    is precisely
    \begin{align*}
    &\#\left\{o\subseteq \Hom((\Q/\Z)^*,T(\pi)) : \substack{\displaystyle \rho_h|_{\C o} = \ind_{L}^K(1_L)\text{ for }L\text{ the fixed field}\\ \displaystyle\text{of }\Stab_{G_K}(e)\text{ for some }e\in o}\right\}\\
    &- \sum_{\substack{o\subseteq \Hom((\Q/\Z)^*,T(\pi))\\ \ind(o) < n}} \underset{s=1/n}{\rm ord} L(\ind(o)s, \rho_h|_{\C o})
    \end{align*}
    by Corollary \ref{cor:twisted_perm_mero}. The first term comes from contributions of the poles of Artin $L$-functions $L(s,\rho)$ at $s=1$, and the second term comes from the possible contribution of nontrivial zeros of Artin $L$-functions.
    
    There is also a possibility that one of the Euler factors for $B(s)$ at the finitely many places $p\mid |T|\disc(\pi)\disc(h)\infty$ has a zero at $s=1/n$, which may partially cancel out the pole (i.e. lowering the order of the pole). This happens if and only if
    \[
    \frac{1}{|H^0(K_p,T)|}\sum_{f_p\in H^1(K_p,T)}\langle f_p,h|_{G_{K_p}}\rangle_p |\disc(f_p*\pi|_{G_{K_p}})|^{-s}
    \]
    has a zero at $s=1/n$, as the factor $\det(I - \rho_h|_{(\C o)^{I_p}}(\Fr_p)p^{-\ind(o)s})$ is known to be zero-free on the right halfplane ${\rm Re}(s) > 0$. In the case that $h=0$, the factor
    \[
    \frac{1}{|H^0(K_p,T)|}\sum_{f_p\in H^1(K_p,T)}\langle f_p,0\rangle_p |\disc(f_p*\pi|_{G_{K_p}})|^{-s} = \frac{1}{|H^0(K_p,T)|}\sum_{f_p\in H^1(K_p,T)}|\disc(f_p*\pi|_{G_{K_p}})|^{-s}
    \]
    is strictly positive, so that $B(s)$ has no real zeros on the region ${\rm Re}(s)>\frac{1}{2a(T)}$ to partially cancel with the pole.
\end{proof}

These results give the necessary ingredients to use a Tauberian theorem: Corollary \ref{cor:meromorphic_continuation} gives an explicit meromorphic continuation, from which bounds in vertical strips can be obtained, and Corollary \ref{cor:poles} gives the location and orders of the poles that will produce the asymptotic main and lower order terms.

\section{Parametrizing the Local Tate Pairings at Tame Places}\label{sec:local_Tate}

The local Tate pairings are examples of a cup product. In order to construct a single representation that unifies these pairings, we will need a generalized version of the cup product.

\begin{definition}
    Let $T$ be a Galois module, $T^* = \Hom(T,\mu)$ the Tate dual, and $\langle,\rangle:T\times T^*\to \mu$ the natural pairing between them. Let $F/K$ be a finite Galois extension containing the fields of definition for $T$ and $T^*$. The \textbf{parametrizing local pairing} is a map
    \[
    [,]:\Hom((\Q/\Z)^*,T) \times (T^*\rtimes \Gal(F/K)) \to \Hom((\Q/\Z)^*,\mu),
    \]
    defined as follows: for each $a\in \Hom((\Q/\Z)^*,T)$ and $g=(g_0,g_1) \in T^*\rtimes \Gal(F/K)$, the pairing is given by
    \[
    [a, g](x) = \langle a(x), g_0\rangle
    \]
    for each $x\in(\Q/\Z)^*$.
\end{definition}

The parametrizing local pairing is built to be a consistent combination of all possible local Tate pairings at the tame places. While the construction does somewhat resemble the usual method of constructing a cup product, it is not immediately obvious why this is the correct notion or how to relate it to local Tate pairings.

The local Tate pairings are made from a cup product which canonically lands in $H^2(K_p,\mu)\cong \Q/\Z$. They are then passed through a non-canonical isomorphism $\Q/\Z \cong \mu$ to get roots of unity. We don't really want the choice of non-canonical isomorphism here to make a difference in our work, and this is consistent with our definition of the parametrizing local pairing. There is a canonical isomorphism between $\Q/\Z$ and $\Hom((\Q/\Z)^*,\mu) = (\Q/\Z)^{**}$ given in the usual way:
\[
r \mapsto (\alpha \mapsto \alpha(r)).
\]
We will allow ourselves to take canonical isomorphisms for granted, and treat the parametrizing local pairing as landing in $\Q/\Z$.

The primary result of this section is that the parametrizing local pairing specializes to (part of) the local Tate pairing at \emph{each} place $p\nmid |T|\infty$ which is unramified in $F/K$.

\begin{theorem}[Parametrizing the Local Tate Pairings]\label{thm:parametrizing}
Let $T$ be a Galois module over $K$, $T^*=\Hom(T,\mu)$ the Tate dual module, and $F/K$ any finite Galois extension that contains the fields of definition for $T$ and $T^*$. For any place $p\nmid |T|\infty$ which is unramified in $F/K$, the following diagram commutes:
\[
\begin{tikzcd}
    H^1(K_p,T)\times Z^1_{ur}(K_p,T^*) \rar{1\times \inf}\dar{\res\times 1} & H^1(K_p,T)\times H^1(K_p,T^*)\rar{\langle,\rangle_p} &\Q/\Z\dar[equals]\\
    H^1(I_p,T)\times Z^1_{ur}(K_p,T^*) \rar{\iota_p^*\times {\rm ev}_{p}} &\Hom((\Q/\Z)^*,T)\times (T^*\rtimes \Gal(F/K)) \rar{[,]} &\Q/\Z
\end{tikzcd}
\]
where
\begin{enumerate}[(a)]
    \item $\inf$ is taken to be the composition of the coboundary quotient $Z^1_{ur}(K_p,T^*) \to H_{ur}^1(K_p,T^*)$ with inflation $H_{ur}^1(K_p,T^*)\to H^1(K_p,T^*)$,
    \item $\langle,\rangle_p$ is the local Tate pairing,
    \item $\iota_p:(\Q/\Z)^* \to I_p^{\rm tame}$ is the canonical quotient map, given by the isomorphism of the prime-to-$p$ part of $(\Q/\Z)^*$ with $I_p^{\rm tame}$, and
    \item ${\rm ev}_{p}:Z^1_{ur}(K_p,T^*) \to T^*\rtimes \Gal(F/K)$ is a modified evaluation map ${\rm ev}_p(h) = h(\Fr_p)\pi(\Fr_p)$, where $\pi:G_K\to \Gal(F/K)$ is the canonical quotient map.
\end{enumerate}

Moreover, the parametrizing local pairing $[,]$ is the unique map with this property.
\end{theorem}

It is important to note that while the top row factors through the coboundary relation, the bottom row \emph{does not}. This is because the left-hand morphism $\res\times 1$ is not surjective, and instead has image
\[
H^1(I_p,T)^{G_{K_p}/I_p}\times Z^1_{ur}(K_p,T^*).
\]
The bottom row restricted to this subgroup will factor through the coboundary relation by a diagram chase. In fact, the composite map
\[
[\iota_p^*,{\rm ev}_p]:H^1(I_p,T)^{G_{K_p}/I_p}\times H^1_{ur}(K_p,T^*) \to \Q/\Z
\]
is equal to the cup product on the Hochschild--Serre spectral sequence
\[
E_2^{p,q}(M):H^p(G_{K_p}/I_p, H^q(I_p,M))\Rightarrow H^{p+q}(K_p,M)
\]
in dimension $E_2^{0,1}(T)\otimes E_2^{1,0}(T^*)\to E_2^{1,1}(\mu)$. The cup product on cohomology is extended to a cup product on spectral sequences of the form
\[
E_2^{p,q}\times E_2^{p',q'} \to E_2^{p+p',q+q'}.
\]
For some references that include a treatment of cup products of spectral sequences, see \cite{hochschild-serre1953,hutchings2011}.

This observation suggests that Theorem \ref{thm:parametrizing} has a short proof using the theory of spectral sequences, and indeed it does. To assist the unfamiliar reader with this proof (and with the use of Theorem \ref{thm:parametrizing}), we have reduced the use of spectral sequences to only the following four properties:
\begin{enumerate}
    \item The definition of the cup product
    \[
    H^1(N,A)^{G/N} \times H^1(G/N,B) \to H^1(G/N, H^1(N,A\otimes B))
    \]
    for $N\normal G$, given by
    \[
    (f\cup h)(gN,n) = f(gN)\otimes h(n).
    \]
    This is induced from the usual cup product on $H^1$, and can be technically be used without refrence to spectral sequences.
    \item Naturality of the cup product with respect to inflation and restriction maps. This is again inherited from the usual cup product on $H^1$.
    \item Uniqueness of the cup product on spectral sequences is used to prove uniqueness of $[,]$ in satisfying Theorem \ref{thm:parametrizing}, although we do not make use of this uniqueness anywhere else in the paper.
    \item The seven term exact sequence extending the inflation-restriction sequence. This sequence is produced from the Hochschild--Serre spectral sequence, although \cite{dekimpe-hartl-wauters2012} give an elemntary treatment with explicit descriptions of each morphism.
\end{enumerate}

\begin{proof}
    The pairing $\langle,\rangle:T\times T^*\to \mu$ always has image in the $e^{\rm th}$ roots of unity $\mu_e$ for $e = |T|$. We will use $\mu_e$ throughout instead of $\mu$, noting that $T^* = \Hom(T,\mu) = \Hom(T,\mu_e)$ as well, as it is convenient that the Galois module $\mu_e$ is unramified at all but finitely many primes.
    
    That $[\iota_p^*,{\rm ev}_p]$ equals the cup product on spectral sequences follows immediately from the definitions. Let $\overline{f}\in H^1(I_p,T)^{G_{K_p}/I_p} = E_2^{0,1}(T)$ and choose a lift $f\in Z^1(I_p,T)$. Likewise, let $h\in Z^1_{ur}(K_p,T^*)$, reducing to $\overline{h}\in H^1_{ur}(K_p,T^*)=E_2^{1,0}(T^*)$. Then
    \begin{align*}
        [\iota_p^*(\overline{f}),{\rm ev}_p(h)](x) &= \langle \overline{f}\iota_p(x), h(\Fr_p)\rangle\\
        &= (f\cup h)(\iota_p(x),\Fr_p)
    \end{align*}
    agrees with the cup product on the level of crossed homomorphisms, landing in
    \begin{align*}
    \Hom(G_{K_p}/I_p,H^1((\Q/\Z)^*,\mu_e)) &= (\iota_p^*)_*\left(\Hom(G_{K_p}/I_p,H^1(I_p^{\rm tame},\mu_e))\right)\\
    &=(\iota_p^*)_*(E_2^{1,1}(\mu_e)),
    \end{align*}
    which is canonically isomorphic to a subgroup of $\Q/\Z$ (specifically the subgroup $\frac{1}{e}\Z/\Z$).

    Define a function
    \[
    \alpha:Z^2(K_p,\mu_e) \to E_2^{1,1}(\mu_e) = H^1(G_{K_p}/I_p,H^1(I_p,\mu_e))
    \]
    by restriction in the first coordinate to $I_p$. We claim that it suffices to prove that $\alpha$ factors through the coboundary relation.
    
    Suppose this is the case and write $\overline{\alpha}$ for the induced map on $H^2(K_p,\mu_e)$. Consider the following diagram:
    \[
    \begin{tikzcd}
    H^1(K_p,T)\times Z^1_{ur}(K_p,T^*) \rar{1\times \inf}\dar{\res\times 1} & H^1(K_p,T)\times H^1(K_p,T^*)\rar{\langle,\rangle_p} &H^2(K_p,\mu_e)\dar{\overline{\alpha}}\\
    H^1(I_p,T)^{G_{K_p}/I_p}\times Z^1_{ur}(K_p,T^*) \rar\dar[hook] & H^1(I_p,T)^{G_{K_p}/I_p}\times H^1_{ur}(K_p,T^*)\rar{\cup} &E_2^{1,1}(\mu_e)\dar{(\iota_p^*)_*}\\
    H^1(I_p,T)\times Z^1_{ur}(K_p,T^*) \rar{\iota_p^*\times {\rm ev}_{p}} &\Hom((\Q/\Z)^*,T)\times (T^*\rtimes \Gal(F/K)) \rar{[,]} &\Q/\Z.
    \end{tikzcd}
    \]
    We just proved that the lower box commutes. The upper box commutes by naturality of the cup product, so all together the existence of $\overline{\alpha}$ proves the diagram in the Theorem statement commutes.

    In fact, not only does $\alpha$ factor through the coboundary relation, $\overline{\alpha}$ is an isomorphism. This follows from the seven term exact sequence associated to the Hochschild--Serre spectral sequence with coefficients in $\mu_e$ \cite{dekimpe-hartl-wauters2012}, given by
    \[
    \begin{tikzcd}
    0 \rar & E_2^{1,0}(\mu_e) \rar & H^1(K_p,\mu_e) \rar & E_2^{0,1}(\mu_e) \arrow{dlll}\\
    E_2^{2,0}(\mu_e) \rar & \ker\left(H^2(G_{K_p},\mu_e) \to E_2^{0,2}(\mu_e)\right) \rar & E_2^{1,1}(\mu_e) \rar & E_2^{3,0}(\mu_e).
    \end{tikzcd}
    \]
    This is, in particular, a lengthening of the inflation-restriction sequence. The map
    \[
    \ker\left(H^2(G_{K_p},\mu_e) \to E_2^{2,0}(\mu_e)\right) \to E_2^{1,1}(\mu_e)
    \]
    is given by restriction of the first coordinate of the representative 2-cocycles, and is exactly what we need to consider. If $\ker(H^2(G_{K_p},\mu_e)\to E_2^{2,0}(\mu_e))$ were equal to all of $H^2(K_p,\mu_e)$, we would be done as this map would be exactly $\overline{\alpha}$.

    We note that $G_{K_p}/I_p = \langle \Fr_p\rangle \cong\hat{Z}$ is a pro-free group, which is the completion of a free group $\Z$ with cohomological dimension $1$. Our assumption that $p\nmid |T|=e$ implies that
    \[
    \begin{tikzcd}
    H^n(I_p,\mu_e) &H^n(I_p^{\rm tame},\mu_e)\lar{\inf} \rar[hook] &H^n(\hat{\Z},\mu_e)
    \end{tikzcd}
    \]
    are isomorphisms. We again know that $\hat{\Z}$ is the completion of $\Z$ which has cohomological dimension $1$. Because $\mu_e$ is finite, we know that $H^{n}(\hat{\Z},\mu_e) \cong H^{n}(\Z,\mu_e)$ so that we can take advantage of the cohomological dimension of $\Z$. All together, this implies $E_2^{2,0}(\mu_e)=0$, $E_2^{0,2}(\mu_e)=0$, and $E_2^{3,0}(\mu_e) = 0$. The seven term short exact sequence then simplifies to
    \[
    \begin{tikzcd}
    0 \rar & E_2^{1,0}(\mu_e) \rar & H^1(K_p,\mu_e) \rar & E_2^{0,1}(\mu_e) \arrow{dlll}\\
    0 \rar & H^2(G_{K_p},\mu)\rar{\overline{\alpha}} & E_2^{1,1}(\mu_e) \rar & 0,
    \end{tikzcd}
    \]
    proving that $\alpha$ does factor through the coboundary relation.

    Uniqueness of $[,]$ follows from the fact that
    \[
    \iota_p^*\times {\rm ev}_p:H^1(I_p,T)\times Z^1_{ur}(K_p,T^*) \to \Hom((\Q/\Z)^*,T)\times (T^*\rtimes \Gal(F/K))
    \]
    is injective with image given by the fiber $q^{-1}(\Fr_p(F/K))$ of the quotient map
    \[
    q:\Hom((\Q/\Z)^*,T)\times (T^*\rtimes \Gal(F/K)) \to \Gal(F/K)
    \]
    The Chebotarev density theorem implies that every element (up to conjugation) of $\Gal(F/K)$ is given by Frobenius of some unramified place, so it follows that these images cover $\Hom((\Q/\Z)^*,T)\times (T^*\rtimes \Gal(F/K))$. Thus, $[,]$ is determined uniquely by the compositions $[\iota_p^*,{\rm ev}_p]$ at finitely many places. We already proved that the value of $[\iota_p^*,{\rm ev}_p]$ was determined by the cup product on the Hochschild--Serre spectral sequence, which is unique.
\end{proof}

\section{Constructing the Galois representation}\label{sec:Galois_rep}

We are now ready to construct the Galois representation $\rho_h$ used in Theorem \ref{thm:euler_factors}.

\begin{definition}\label{def:twisted_perm_rep}
    Let $h\in H^1(K,T^*)$, choose a representative crossed homomorphism for $h$, let $F/K$ be a finite Galois extension containing the fields of definition of $T$ and $T^*$, and let $\pi:G_K\to \Gal(F/K)$ be the canonical quotient map. Define the Galois representation
    \[
    \rho_h: G_K \to {\rm GL}(\C[\Hom((\Q/\Z)^*,T)])
    \]
    by the induced action on each of the basis elements $a\in \Hom((\Q/\Z)^*,T)$
    \[
    \rho_h(g)a = [g.a, (h*\pi)(g)] (g.a),
    \]
    where we choose an identification of $\Hom((\Q/\Z)^*,\mu)\cong\Q/\Z$ with $\mu$ so that the coefficient is a complex number.
\end{definition}

Firstly, we should confirm that $\rho_h$ is really a homomorphism. The parametrizing local pairing is not quite a bilinear map, instead it satisfies a type of twisted bilinearity. We prove this in the following lemma, which we can use to check that $\rho_h$ is a homomorphism.

\begin{lemma}\label{lem:twisted_bilinear}
    Let $a,b\in \Hom((\Q/\Z)^*,T)$ and $f=(f_0,f_1),g=(g_0,g_1)\in T^*\rtimes \Gal(F/K)$. Then
    \begin{enumerate}
        \item[(i)] $[ab,f] = [a,f][b,f]$, and
        \item[(ii)] $[a,fg] = [a,f][f_1^{-1}.a,g]$.
    \end{enumerate}
\end{lemma}

Using the second property, we can confirm that
\begin{align*}
    \rho_h(xy)a &= [xy.a, (h*\pi)(xy)] (xy.a)\\
    &=[xy.a, (h*\pi)(x)(h*\pi)(y)] (xy.a)\\
    &=[xy.a, (h*\pi)(x)][(h*\pi)(x)_1^{-1}.(xy.a),(h*\pi)(y)] (xy.a).
\end{align*}
The action of $G_K$ on $\Hom((\Q/\Z)^*,T)$ factors through $\Gal(F/K)$, which implies $(h*\pi)(x)_1^{-1}$ acts in the same way as $x^{-1}$. Thus
\begin{align*}
    \rho_h(xy)a &=[xy.a, (h*\pi)(x)][x^{-1}.(xy.a),(h*\pi)(y)] (xy.a)\\
    &=[xy.a, (h*\pi)(x)][y.a,(h*\pi)(y)] (xy.a)\\
    &=\rho_h(x)[y.a,(h*\pi)(y)] (y.a)\\
    &=\rho_h(x)\rho_h(y)a.
\end{align*}
Thus, $\rho_h$ really is a Galois representation.

\begin{proof}[Proof of Lemma \ref{lem:twisted_bilinear}]
    We may prove the lemma directly. Let $x\in (\Q/\Z)^*$. Then
    \begin{align*}
        [ab,f](x) &= \langle (ab)(x), f_0\rangle\\
        &= \langle a(x)b(x), f_0\rangle\\
        &= \langle a(x), f_0\rangle\langle b(x), f_0\rangle\\
        &=[a,f](x)[b,f](x).
    \end{align*}

    For the other coordinate, we evaluate
    \begin{align*}
        [a,fg](x) &= \langle a(x), (fg)_0\rangle.
    \end{align*}
    The semi-direct product structure implies that
    \begin{align*}
        ((fg)_0, (fg)_1) &= (f_0,f_1)(g_0,g_1)\\
        &=(f_0 (f_1.g_0), f_1g_1).
    \end{align*}
    The action on $g_0\in T^*$ is given by
    \[
    (f_1.g_0)(y) = g_0(f_1^{-1}.y)^{\chi(f_1)}
    \]
    for $\chi:G_K\to \hat{\Z}^{\times}$ the cyclotomic character. Thus,
    \begin{align*}
        [a,fg](x) &= \langle a(x), f_0 (f_1.g_0)\rangle\\
        &=\langle a(x), f_0\rangle\langle a(x), f_1.g_0\rangle\\
        &=\langle a(x), f_0\rangle\langle f_1^{-1}.(a(x)), g_0\rangle^{\chi(f_1)}.
    \end{align*}
    However, $\Gal(F/K)$ acts on $\Hom((\Q/\Z)^*,T)$ in a similar way:
    \[
    (f_1^{-1}.a)(x) = f_1^{-1}.(a(x^{\chi(f_1)})).
    \]
    Moving the cyclotomic character all the way to the inside, we compute
    \begin{align*}
        [a,fg](x) &= \langle a(x), f_0\rangle\langle (f_1^{-1}.a)(x), g_0\rangle\\
        &= [a,f](x)[f_1^{-1}.a,g](x).
    \end{align*}
\end{proof}

This is all the setup we need to prove Theorem \ref{thm:euler_factors}.

\begin{proof}[Proof of Theorem \ref{thm:euler_factors}]
Choose a representative crossed homomorphism for $h\in H^1(K,T^*)$, and let $p$ be a place of $K$ for which $p\nmid |T|\disc(\pi)\disc(h)\infty$. In particular, if $F/K$ is a finite Galois extension containing the fields of definition for $T$ and $T^*$ then $p$ is unramified in $F/K$.

For each $f_p\in H^1(K_p,T)$, Theorem \ref{thm:parametrizing} implies that
\[
\langle f_p, h|_{G_{K_p}} \rangle_p = [f_p|_{I_p}\circ \iota_p, (h*\pi)(\Fr_p)],
\]
which depends only on the restriction $f_p|_{I_p}$. The same is true for the discriminant, as $p\nmid \disc(\pi)$ implies
\[
|\disc(f_p*\pi|_{G_{K_p}})|^{-s} = p^{-\ind(f_p(\tau))s}
\]
for any generator $\tau\in I_p^{\rm tame}$. Noting that the size of the set of coboundaries is $|B^1(K_p,T)| = \frac{|T|}{|H^0(K_p,T)|}$ and that the fibers of $\res:Z^1(K_p,T)\to Z^1(I_p,T)$ have size $|T|$, this implies
\begin{align*}
    &\frac{1}{|H^0(K_p,T)|}\sum_{f_p\in H^1(K_p,T)}\langle f_p,h|_{G_{K_p}}\rangle_p |\disc(f_p*\pi|_{G_{K_p}})|^{-s}\\
    &= \frac{1}{|T|}\sum_{f_p\in Z^1(K_p,T)}[f_p|_{I_p}\circ \iota_p, (h*\pi)(\Fr_p)] p^{-\ind(f_p(\tau))s}\\
    &= \sum_{f_p\in \res_{I_p}(Z^1(K_p,T))}[f_p\circ \iota_p, (h*\pi)(\Fr_p)] p^{-\ind(f_p(\tau))s}.
\end{align*}
Notice that because $F/K$ is unramified at $p$, $Z^1(I_p,T) = \Hom(I_p,T)$. The restriction is then
\[
\res_{I_p}(Z^1(K_p,T)) = Z^1(I_p,T)^{G_{K_p}/I_p} = \Hom(I_p,T)^{G_{K_p}/I_p}.
\]
We pull back along the map $\iota_p:(\Q/\Z)^*\to I_p^{\rm tame}$ to remove most of the dependence on $p$, showing that this sum is equal to
\[
\sum_{a\in \Hom((\Q/\Z)^*,T)^{G_{K_p}/I_p}}[a, (h*\pi)(\Fr_p)] p^{-\ind(a)s} = \sum_{\substack{a\in \Hom((\Q/\Z)^*,T)\\\Fr_p.a=a}}[a, (h*\pi)(\Fr_p)] p^{-\ind(a)s}.
\]
Partitioning the sum by the $G_K$-orbits $o\subseteq \Hom((\Q/\Z)^*,T)$, this is equal to
\[
\sum_{o\subseteq \Hom((\Q/\Z)^*,T)}\left(\sum_{\substack{a\in o\\\Fr_p.a=a}}[a, (h*\pi)(\Fr_p)]\right)p^{-\ind(o)s}.
\]
The action of $G_K$ on $\Hom((\Q/\Z)^*,T)$ is given by the permutation representation $\overline{\rho_h}$, so that $\Fr_p.a=a$ if and only if $\overline{\rho_h}(\Fr_p)a = a$. By definition, this is equivalent to $a$ being an eigenvector of $\rho_h(\Fr_p)$. The associated eigenvalue is necessarily $[a,(h*\pi)(\Fr_p)]$. Thus,
\[
\sum_{\substack{a\in o\\\Fr_p.a=a}}[a, (h*\pi)(\Fr_p)] = \text{the sum of the eigenvalues of }\rho_h|_{\C o}(\Fr_p) = {\rm tr}(\rho_h|_{\C o}(\Fr_p)).
\]
Thus, we conclude that the Euler factor is given by
\[
\sum_{o\subseteq \Hom((\Q/\Z)^*,T)}{\rm tr}(\rho_h|_{\C o}(\Fr_p))p^{-\ind(o)s}.
\]
\end{proof}

\section{A Power Saving Tauberian Theorem}\label{sec:Tauberian}
We will use the following Tauberian theorem:

\begin{theorem}\label{thm:xi+1}
    Let $a_n$ be a sequence of complex numbers and $\widehat{a}_n$ a sequence of real numbers with $|a_n|\le \widehat{a}_n$. Let $L(s) = \sum a_n n^{-s}$ and $\widehat{L}(s) = \sum \widehat{a}_n n^{-s}$ be the corresponding Dirichlet series with abscissa of absolute convergence $\sigma_a$ and $\widehat{\sigma}_a$ respectively. Suppose that
    \begin{enumerate}[(a)]
        \item $L(s)$ and $\widehat{L}(s)$ have a meromorphic continuation to the right halfplane ${\rm Re}(s) > \widehat{\sigma}_a - \delta$ with at most finitely many poles, and
        \item $|L(s)|,|\widehat{L}(s)| \ll_{\epsilon} (1+|t|)^{\xi+\epsilon}$ for some $\xi\ge 0$, each $\epsilon > 0$, and each $s$ with ${\rm Re}(s) > \widehat{\sigma}_a - \delta$ and sufficiently large imaginary part.
    \end{enumerate}
    Then
    \begin{align*}
        \sum_{n\le X} a_n &= \sum_{j=1}^r \underset{s=s_j}{\rm Res} \left(L(s)\frac{X^s}{s}\right) + O\left(X^{\widehat{\sigma}_a - \frac{\delta}{\xi + 1}+\epsilon}\right),
    \end{align*}
    where $s_1,s_2,...,s_r$ is the sequence of poles of $L(s)/s$ with ${\rm Re}(s_j) > \widehat{\sigma}_a - \frac{\delta}{\xi + 1}$.
\end{theorem}

Landau originally proved a power savings of the form $O(X^{\widehat{\sigma}_a - \frac{\delta}{\lfloor \xi\rfloor + 2}+\epsilon})$ in \cite{landau1915} under the condition that $L(s)$ is meromorphic on all of $\C$, where $\lfloor \xi\rfloor$ is the floor of $\xi$. In a more recent treatment, Roux \cite[Theorems 13.3 and 13.8]{roux2011} relaxed Landau's hypotheses to produce the power savings $O(X^{\widehat{\sigma}_a - \frac{\delta}{\lfloor \xi\rfloor + 2}+\epsilon})$ under the same hypotheses as Theorem \ref{thm:xi+1} above.

The improvement from $\widehat{\sigma}_a - \frac{\delta}{\lfloor \xi\rfloor + 2}+\epsilon$ to $\widehat{\sigma}_a - \frac{\delta}{\xi + 1}+\epsilon$ is most substantial when $\xi$ is small. In particular, when $\xi=0$ Theorem \ref{thm:xi+1} gives a power savings $O(X^{\widehat{\sigma}_a-\delta+\epsilon})$ that agrees with the region of meromorphicity. This should be expected, as $\xi=0$ means that the vertical bounds of $L(s)$ are negligible and should not affect the asymptotic growth rate.

Our proof is very similar to Roux's, by shifting the contour integral in Perron's formula. Roux uses the following version of Perron's formula:
\[
\int_{0}^X\int_0^{x_1}\cdots \int_0^{x_{k-1}} \sum_{n<x_k} a_n dx_{k}\cdots dx_2dx_1 = \frac{1}{2\pi i} \int_{c-i\infty}^{c+i\infty} L(s) \frac{X^{s+k}}{s(s+1)\cdots (s+k)}ds,
\]
for some $c>\widehat{\sigma}_a$. Roux chooses $k=\lfloor \xi\rfloor + 1$ and proves that most pieces of the shifted contour are of negligible size with this choice when $T$ is sufficiently large. We improve this argument by dealing with all pieces of the shifted contour directly, and by choosing $T$ optimally as a function of $k$. The size of the error term produced by this choice decreases as $k\to \infty$, with $O(X^{\widehat{\sigma}_a-\frac{\delta}{\xi+1} + \epsilon})$ being the limit.

We opt to present the proof of Theorem \ref{thm:xi+1} in its entirety, though a number of portions are the same as those appearing in \cite{roux2011}. We do this for the sake of completeness - it will help to clarify where our improvements occur, and it will allow readers to see an English translation of the components of \cite{roux2011}. We also give more streamlined proofs for certain components.

\subsection{The Difference Operators}

Let $N(x) = N_0(x) = \sum_{n<x} a_n$ and define
\[
N_k(x) = \int_0^x N_{k-1}(u)du.
\]
We make a similar definition for $\widehat{N}_k(x)$ for the sequence $\widehat{a}_n$.

In order to prove Theorem \ref{thm:xi+1}, we will shift the contour in Perron's formula
\[
N_k(x) = \int_{0}^X\int_0^{x_1}\cdots \int_0^{x_{k-1}} \sum_{n<x_k} a_n dx_{k}\cdots dx_2dx_1 = \frac{1}{2\pi i} \int_{c-i\infty}^{c+i\infty} L(s) \frac{X^{s+k}}{s(s+1)\cdots (s+k)}ds.
\]
We need to relate this back to $N(x)$, which we will do via difference operators.

Define
\begin{align*}
    \Delta_y f&: x \mapsto f(x+y)-f(x) & \Delta_y^{(k)} = \underbrace{\Delta_y\circ\cdots \circ\Delta_y}_{k}.
\end{align*}
By definition, the $k^{\rm th}$ derivative $N_k^{(k)}(x) = N(x)$ exists. We will approximate this equality using the $k$-fold difference operator:
\begin{align}\label{eq:Delta_approx_N(x)}
y^{-k}\Delta_y^{(k)} N_k(x) \approx N(x).
\end{align}
The following lemma is an immediate consequence of the Fundamental Theorem of Calculus:

\begin{lemma}\label{lem:FTC}
    Define
    \begin{align*}
        \widetilde{\Delta}_yf&: x\mapsto \int_x^{x+y} f(u)du & \widetilde{\Delta}_y^{(k)} = \underbrace{\widetilde{\Delta}_y\circ\cdots \circ\widetilde{\Delta}_y}_{k}.
    \end{align*}
    Then $\Delta_y^{(k)} f(x) = \widetilde{\Delta}_y^{(k)} f^{(k)}(x)$.
\end{lemma}

This gives us an alternative way to write
\[
y^{-k}\Delta_y N_k(x) = y^{-k}\widetilde{\Delta}_y^{(k)}N_0(x) = y^{-k}\int_x^{x+y}\int_{x_1}^{x_1+y} \cdots \int_{x_{k-1}}^{x_{k-1}+y} N(x_k) dx_k\cdots dx_2 dx_1.
\]
By making strategic choices between the difference $\Delta_y$ and the integral difference $\widetilde{\Delta}_y$ at various steps in the proof, we are able to better optimize the error bounds.

The approximation (\ref{eq:Delta_approx_N(x)}) is concretely accomplished via \cite[Proposition 16.16]{roux2011} and \cite[Proposition 16.19]{roux2011}. Roux's results can be cited without modification, but we opt to provide streamlined, translated proofs for the sake of completeness.

\begin{proposition}[{\cite[Proposition 16.16]{roux2011}}]\label{prop:roux16.16}
    Suppose $a_n\ge 0$ and $N(x) = \sum_{n<x} a_n$. Then
    \[
    y^{-k} \Delta_y^{(k)} N_k(X-ky) \le N(X) \le y^{-k} \Delta_y^{(k)}N_k(X)
    \]
\end{proposition}
This implies
\begin{align*}
    y^{-k} \Delta_y^{(k)}N_k(X) &= N(X) + O\left(y^{-k}\Delta_y^{(k)}[N_k(X)-N_k(X-ky)]\right).
\end{align*}
in the case that $a_n\ge 0$.

\begin{proof}[Proof of Proposition \ref{prop:roux16.16}]
    The fact that $a_n\ge 0$ implies that $N(X)$ is an increasing function. Thus,
    \begin{align*}
        \Delta_y^{(k)}N_k(X) &= \widetilde{\Delta}_y^{(k)}N(X)\\
        &=\int_X^{X+y}\int_{x_1}^{x_1+y}\cdots \int_{x_{k-1}}^{x_{k-1}+y} N(x_k)dx_k\cdots dx_2dx_1\\
        &\ge \int_X^{X+y}\int_{x_1}^{x_1+y}\cdots \int_{x_{k-1}}^{x_{k-1}+y} N(X)dx_k\cdots dx_2dx_1\\
        &= y^kN(X),
    \end{align*}
    as the bounds on the integrals force $x_k\ge \cdots \ge x_1\ge X$.

    The other inequality is similar, as
    \begin{align*}
        \Delta_y^{(k)}N_k(X-ky) &= \widetilde{\Delta}_y^{(k)}N(X-ky)\\
        &=\int_{X-ky}^{X-(k-1)y}\int_{x_1}^{x_1+y}\cdots \int_{x_{k-1}}^{x_{k-1}+y} N(x_k)dx_k\cdots dx_2dx_1\\
        &\le \int_{X-ky}^{X-(k-1)y}\int_{x_1}^{x_1+y}\cdots \int_{x_{k-1}}^{x_{k-1}+y} N(X)dx_k\cdots dx_2dx_1\\
        &= y^kN(X),
    \end{align*}
    as the bounds for the integrals force $x_k\le x_{k-1}+y \le \cdots \le x_1+(k-1)y \le X$.
\end{proof}

We next use $\widehat{N}(X)$ to give a bound for $y^{-k}\Delta_y^{(k)}N_k(X) - N(X)$ when $a_n$ is a general sequence of complex numbers. The following proposition plays an analogous role to \cite[Proposition 16.19]{roux2011}, although it is more streamlined. The proof is shorter (not requiring an induction argument), and the result is more immediately applicable to proving Theorem \ref{thm:xi+1}.

\begin{proposition}\label{prop:roux16.19+}
    Suppose $|a_n|\le \widehat{a}_n$ and set $N(x) = \sum_{n<x} a_n$ and $\widehat{N}(x) = \sum_{n<x}\widehat{a}_n$. Then
    \[
    y^{-k} \Delta_y^{(k)}N_k(X) = N(X) + O\left(y^{-k} \Delta_y^{(k)}\widehat{N}_k(X) - \widehat{N}(X)\right).
    \]
\end{proposition}

Thus, the error given by Proposition \ref{prop:roux16.16} for $\widehat{N}(X)$ can also be used for $N(X)$.

\begin{proof}[Proof of Proposition \ref{prop:roux16.19+}]
    By Lemma \ref{lem:FTC},
    \begin{align*}
        y^{-k} \Delta_y^{(k)}N_k(X) &= y^{-k} \widetilde{\Delta}_y^{(k)}N(X)\\
        &=y^{-k}\int_X^{X+y}\int_{x_1}^{x_1+y}\cdots \int_{x_{k-1}}^{x_{k-1}+y} N(x_k)dx_k\cdots dx_2dx_1\\
        &=N(X) + y^{-k}\int_X^{X+y}\int_{x_1}^{x_1+y}\cdots \int_{x_{k-1}}^{x_{k-1}+y} N(x_k) - N(X) dx_k\cdots dx_2dx_1\\
        &=N(X) + O\left(y^{-k}\int_X^{X+y}\int_{x_1}^{x_1+y}\cdots \int_{x_{k-1}}^{x_{k-1}+y} |N(x_k) - N(X)| dx_k\cdots dx_2dx_1\right).
    \end{align*}
    We remark that for $x_k\ge X$,
    \begin{align*}
        |N(x_k) - N(X)| &= \left\lvert\sum_{X\le n<X_k}a_n\right\rvert\\
        &\le \sum_{X\le n<x_k}|a_n|\\
        &\le \sum_{X\le n<x_k}\widehat{a}_n\\
        &=\widehat{N}(x_k)-\widehat{N}(X).
    \end{align*}
    Thus
    \begin{align*}
        y^{-k} \Delta_y^{(k)}N_k(X) &= N(X) + O\left(y^{-k}\int_X^{X+y}\int_{x_1}^{x_1+y}\cdots \int_{x_{k-1}}^{x_{k-1}+y} \widehat{N}(x_k) - \widehat{N}(X) dx_k\cdots dx_2dx_1\right)\\
        &= N(X) + O\left(y^{-k} \Delta_y^{(k)}\widehat{N}_k(X) - \widehat{N}(X)\right).
    \end{align*}
\end{proof}

\subsection{Shifting the Contour Integral}

We put the contour shifting argument under a single proposition:

\begin{proposition}\label{prop:contour_shift}
    Let $a_n$ be a sequence of complex numbers with generating Dirichlet series $L(s) = \sum a_n n^{-s}$ with abscissa of absolute convergence $\le \sigma_a$. Suppose that
    \begin{enumerate}[(a)]
        \item $L(s)$ has a meromorphic continuation to the right halfplane ${\rm Re}(s) > \sigma_a - \delta$ with at most finitely many poles, and
        \item $|L(s)|\ll_{\epsilon} (1+|t|)^{\xi+\epsilon}$ for some $\xi\ge 0$, each $\epsilon > 0$, and each $s$ with ${\rm Re}(s) > \sigma_a - \delta$ and sufficiently large imaginary part.
    \end{enumerate}
    Then setting $y = X^{1-\frac{(k-\max\{\xi-1,0\})\delta}{(k+1)\xi+k-\max\{\xi-1,0\}}}$ implies that for all $k$ sufficiently large depending only on $\delta$, $\xi$, and $\epsilon>0$
    \begin{align*}
        y^{-k}\Delta_y^{(k)} N_k(X) &= \sum_{j=1}^r \underset{s=s_j}{\rm Res} \left(L(s)\frac{X^s}{s}\right) + O_{\delta,\xi,\epsilon}\left(X^{\sigma_a - \frac{\delta}{\xi + 1}+\epsilon}\right),
    \end{align*}
    where $s_1,s_2,...,s_r$ is the sequence of poles of $L(s)/s$ with ${\rm Re}(s_j) > \sigma_a - \frac{\delta}{\xi + 1}$.
\end{proposition}

This proposition applies to both $N(X)$ and $\widehat{N}(X)$. The smoother counting in $N_k(X)$ is what allows this Proposition to use only $\sigma_a$ instead of $\widehat{\sigma}_a$ in the error bounds.

\begin{proof}
    Consider Perron's formula
    \[
    N_k(X) = \frac{1}{2\pi i}\int_{c-i\infty}^{c+i\infty} L(s)\frac{X^{s+k}}{s(s+1)\cdots(s+k)}ds,
    \]
    which holds for any $c > \sigma_a$. We will take $c=\sigma_a+\epsilon$ (one often takes $c=\sigma_a + \frac{1}{\log X}$ for this type of contour shifting argument, but this will be unnecessary for our purposes).

    We push the countour integral to the line ${\rm Re}(s) = \sigma_a-\delta+\epsilon$ (with $\epsilon>0$ chosen small so that $L(s)$ does not have any poles on this line) by taking
    \[
    \int_{\sigma_a + \epsilon-i\infty}^{\sigma_a + \epsilon+i\infty} = \int_C + \int_{\sigma_a + \epsilon-i\infty}^{\sigma_a + \epsilon-iT} + \int_{\sigma_a + \epsilon-iT}^{\sigma_a-\delta+\epsilon-iT}+\int_{\sigma_a-\delta+\epsilon-iT}^{\sigma_a-\delta+\epsilon+iT} + \int_{\sigma_a-\delta+\epsilon+iT}^{\sigma_a + \epsilon+iT} + \int_{\sigma_a + \epsilon+iT}^{\sigma_a + \epsilon+i\infty},
    \]
    where $C$ is the closed contour given by a box $\sigma_a + \epsilon-iT$ to $\sigma_a + \epsilon+iT$ to $\sigma_a-\delta+\epsilon + iT$ to $\sigma_a-\delta+\epsilon - iT$.

    Set
    \[
    \Phi(X) = \frac{1}{2\pi i}\int_C L(s) \frac{X^{s+k}}{s(s+1)\cdots (s+k)} ds.
    \]
    This is differentiable as a function of $X$, with
    \begin{align*}
    \Phi^{(k)}(X) &= \frac{1}{2\pi i}\int_C L(s) \frac{X^{s}}{s} ds\\
    &=\sum_{j=1}^{R} \underset{s=s_j}{\rm Res} \left(L(s)\frac{X^s}{s}\right)\\
    &=\sum_{j=1}^{r} \underset{s=s_j}{\rm Res} \left(L(s)\frac{X^s}{s}\right) + O\left(X^{\sigma_a - \frac{\delta}{\xi+1} + \epsilon}\right),
    \end{align*}
    where $s_1,s_2,...,s_R$ are the poles of $L(s)$ with real part $>\sigma_a-\delta+\epsilon$, ordered by increasing real parts. The final equality follows from the fact that
    \[
    \underset{s=s_j}{\rm Res} \left(L(s)\frac{X^s}{s}\right) = O\left(X^{{\rm Re}(s_j) + \epsilon}\right).
    \]
    In particular,
    \begin{align*}
        \Delta_y^{(k)} \Phi(X) &= \widetilde{\Delta}_y^{(k)} \Phi^{(k)}(X)\\
        &= y^k\Phi^{(k)}(X) + \widetilde{\Delta}_y^{(k)}\left[ \Phi^{(k)}(u)-\Phi^{(k)}(X)\right]_{u=X}\\
        &= y^k\Phi^{(k)}(X) + O\left(y^{k+1}\max_{X\le v\le X+ky}\{\Phi^{(k+1)}(v)\}\right).
    \end{align*}
    Knowing that all the terms of $\Phi^{(k)}(X)$ are of the form $X^a(\log X)^b$, we can differentiate this function directly and conclude that $\Phi^{(k)}(v) \ll v^{\sigma_a-1}$. Given the assumption that $y<X$, this implies
    \begin{align*}
        \Delta_y^{(k)} \Phi(X) &= y^k\Phi^{(k)}(X) + O_k\left(y^{k+1} X^{\sigma_a-1}\right).
    \end{align*}

    We now bound the remaining contour integrals. The Dirichlet series $L(s)$ converges absolutely in the region ${\rm Re}(s) > \sigma_a$. In particular, this implies
    \[
    |L(s)| = O_{\epsilon}(1)
    \]
    on the line ${\rm Re}(s) = \sigma_a + \epsilon$. This gives a bound
    \begin{align*}
        \left\lvert\Delta_y^{(k)}\int_{\sigma_a + \epsilon+iT}^{\sigma_a+\epsilon+i\infty}L(s)\frac{X^{s+k}}{s(s+1)\cdots (s+k)}\ ds\right\rvert &\ll \int_{\sigma_a + \epsilon+iT}^{\sigma_a+\epsilon+i\infty}\frac{|\Delta_y^{(k)} X^{s+k}|}{|s(s+1)\cdots (s+k)|}ds\\
        &\ll 2^k (X+ky)^{\sigma_a+k+\epsilon} \int_T^\infty \frac{1}{t^{k+1}}\ dt\\
        &\ll_k X^{\sigma_a+k+\epsilon}T^{-k}.
    \end{align*}
    A similar bound holds for $\int_{\sigma_a+\epsilon-i\infty}^{\sigma_a+\epsilon-iT}$ and for
    \begin{align*}
        \left\lvert\Delta_y^{(k)}\int_{\sigma_a -\delta + \epsilon\pm iT}^{\sigma_a+\epsilon\pm iT}L(s)\frac{X^{s+k}}{s(s+1)\cdots (s+k)}\ ds\right\rvert &= \left\lvert\int_{\sigma_a -\delta + \epsilon\pm iT}^{\sigma_a+\epsilon\pm iT}L(s)\frac{\Delta_y^{(k)}X^{s+k}}{s(s+1)\cdots (s+k)}\ ds\right\rvert\\
        &\ll_k X^{\sigma_a+k+\epsilon} \int_{\sigma_a -\delta + \epsilon\pm iT}^{\sigma_a+\epsilon\pm iT} \frac{(1+|t|)^{\xi+\epsilon}}{|s|^{k+1}}\ ds\\
        &\ll_{\delta,k} X^{\sigma_a+k+\epsilon} T^{\xi-k-1+\epsilon}.
    \end{align*}
    We combine these together as
    \[
    O_{\delta,k}(X^{\sigma_a+k+\epsilon} T^{\max\{\xi - 1,0\}-k+\epsilon}).
    \]

    For the remaining contour integral, we will need to use Lemma \ref{lem:FTC} to switch from $\Delta_y$ and $\widetilde{\Delta}_y$. We bound
    \begin{align*}
        \left\lvert\Delta_y^{(k)}\int_{\sigma_a -\delta + \epsilon-iT}^{\sigma_a-\delta+\epsilon+iT}L(s)\frac{X^{s+k}}{s(s+1)\cdots (s+k)}\ ds\right\rvert &= \left\lvert\int_{\sigma_a -\delta + \epsilon-iT}^{\sigma_a-\delta+\epsilon+iT}L(s)\frac{\Delta_y^{(k)}X^{s+k}}{s(s+1)\cdots (s+k)}\ ds\right\rvert\\
        &=\left\lvert\int_{\sigma_a -\delta + \epsilon-iT}^{\sigma_a-\delta+\epsilon+iT}L(s)\frac{\widetilde{\Delta}_y^{(k)}(s+k)\cdots (s+1)X^{s}}{s(s+1)\cdots (s+k)}\ ds\right\rvert\\
        &=\left\lvert\int_{\sigma_a -\delta + \epsilon-iT}^{\sigma_a-\delta+\epsilon+iT}L(s)\frac{\widetilde{\Delta}_y^{(k)}X^{s}}{s}\ ds\right\rvert\\
        &\ll y^k X^{\sigma_a-\delta+\epsilon} \int_{\sigma_a-\delta+\epsilon-iT}^{\sigma_a-\delta+\epsilon+iT} \frac{(1+|t|)^{\xi+\epsilon}}{|s|}\ ds\\
        &\ll y^k X^{\sigma_a-\delta+\epsilon} T^{\xi+\epsilon}.
    \end{align*}
    
    \begin{remark}
        This bound is where our proof differs from Roux's. Roux directly bounds $\Delta_y^{(k)} X^{s+k} \ll_k X^{\sigma_a+k+\epsilon}$ instead of switching to $\widetilde{\Delta}_y$, showing that this piece of the contour is
        \[
        \ll_k X^{\sigma_a+k+\epsilon}(T^{\xi-k-1+\epsilon}+O(1)).
        \]
        By choosing $k > \xi+1$, each term with a $T$ decays in $T$. Roux can then choose $T$ sufficiently large in terms of $X$ so that the error is dominated by only
        \[
        O(y^kX^{\sigma_a-1}) + O(X^{\sigma_a+k+\epsilon}).
        \]
        Our bound allows us to choose an optimum value for $T$, resulting in the improved power savings.
    \end{remark}

    Putting these together implies that
    \begin{align*}
        y^{-k}\Delta_y^{(k)} N(t) &= \Phi^{(k)}(X) + O_k(yX^{\sigma_a-1}) + O_{\delta,k}( y^{-k} X^{\sigma_a+k+\epsilon} T^{\max\{\xi-1,0\}-k+\epsilon}) + O(X^{\sigma_a-\delta+\epsilon}T^{\xi+\epsilon})
    \end{align*}
    We now set $T=X^\frac{(k+1)\delta}{(k+1)\xi+k-\max\{\xi-1,0\}}$ and $y=X^{1-\frac{(k-\max\{\xi-1,0\})\delta}{(k+1)\xi+k-\max\{\xi-1,0\}}}$. The error terms are then
    \begin{align*}
        O(y X^{\sigma_a - 1}) &= O(X^{\sigma_a - \frac{(k-\max\{\xi-1,0\})\delta}{(k+1)\xi+k-\max\{\xi-1,0\}}+\epsilon})\\
        O(y^{-k}X^{\sigma_a+k+\epsilon}T^{\max\{\xi-1,0\}-k+\epsilon}) &= O(X^{\sigma_a - \frac{(k-\max\{\xi-1,0\})\delta}{(k+1)\xi+k-\max\{\xi-1,0\}}+\epsilon})\\
        O(X^{\sigma_a-\delta+\epsilon}T^{\xi+\epsilon}) &= O(X^{\sigma_a - \frac{(k-\max\{\xi-1,0\})\delta}{(k+1)\xi+k-\max\{\xi-1,0\}}+\epsilon}).
    \end{align*}
    In fact, these choices of $T$ and $y$ are optimal when $k$ is sufficiently large as the first error term is increasing in $y$, the second is decreasing in $y$ and $T$ for $k > \xi - 1$, and the third is increasing in $T$.

    It is clear that
    \[
    \lim_{k\to \infty} \sigma_a - \frac{(k-\max\{\xi-1,0\})\delta}{(k+1)\xi+k-\max\{\xi-1,0\}} = \sigma - \frac{\delta}{\xi+1},
    \]
    so there exists a sufficiently large $k$ depending only on $\delta$, $\xi$, and $\epsilon>0$ such that
    \[
    \sigma_a - \frac{(k-\max\{\xi-1,0\})\delta}{(k+1)\xi+k-\max\{\xi-1,0\}} \le \sigma - \frac{\delta}{\xi+1} + \epsilon.
    \]
    Thus, there exists a $k$ sufficiently large depending only on $\delta$, $\xi$, and $\epsilon>0$ such that
    \[
    y^{-k}\Delta_y^{(k)} N(X) = \Phi^{(k)}(X) + O(X^{\sigma_a - \frac{\delta}{\xi+1}+\epsilon}).
    \]
    Given that $\Phi^{(k)}(X)$ equals the sum of residues in the Proposition statement, this concludes the proof.
\end{proof}

\subsection{The Proof of Theorem \ref{thm:xi+1}}

We first apply Proposition \ref{prop:roux16.16} and Proposition \ref{prop:contour_shift} to $\widehat{N}(X)$. Let $y = X^{1-\frac{(k-\max\{\xi-1,0\})\delta}{(k+1)\xi+k-\max\{\xi-1,0\}}} = o(X)$ and $k$ be sufficiently large in terms of $\delta$, $\xi$, and $\epsilon > 0$ to satisfy Proposition \ref{prop:contour_shift}. As in the proof of Proposition \ref{prop:contour_shift}, take $\widehat{\Phi}^{(k)}(X)$ to be equal to the sum of residues of $\widehat{L}(s)/s$. Then
\begin{align*}
    y^{-k} \Delta_y^{(k)}\widehat{N}_k(X) &= \widehat{N}(X) + O\left(y^{-k}\Delta_y^{(k)}[\widehat{N}_k(X)-\widehat{N}_k(X-ky)]\right)\\
    &= \widehat{N}(X) + O\left(\widehat{\Phi}^{(k)}(X) - \widehat{\Phi}^{(k)}(X-ky)\right) + O\left(X^{\widehat{\sigma}_a-\frac{\delta}{\xi+1}+\epsilon}\right)\\
    &= \widehat{N}(X) + O\left(ky\underset{X-ky\le v\le X}{\max}\widehat{\Phi}^{(k+1)}(v)\right) + O\left(X^{\widehat{\sigma}_a-\frac{\delta}{\xi+1}+\epsilon}\right)\\
    &= \widehat{N}(X) + O(yX^{\widehat{\sigma}_a-1}) + O\left(X^{\widehat{\sigma}_a-\frac{\delta}{\xi+1}+\epsilon}\right)\\
    &= \widehat{N}(X) + O\left(X^{\widehat{\sigma}_a-\frac{\delta}{\xi+1}+\epsilon}\right).
\end{align*}
This proves Theorem \ref{thm:xi+1} in the case that $a_n = \widehat{a}_n$ are nonnegative real numbers.

For the same choices of $y$ and $k$, Proposition \ref{prop:roux16.19+} then implies
\begin{align*}
    y^{-k} \Delta_y^{(k)}N_k(X) &= N(X) + O\left(y^{-k} \Delta_y^{(k)}\widehat{N}_k(X) - \widehat{N}(X)\right)\\
    &= N(X) + O\left(X^{\widehat{\sigma}_a-\frac{\delta}{\xi+1}+\epsilon}\right),
\end{align*}
concluding the proof.

\subsection{The region of meromorphic continuation and the power saving bounds cannot be conflated}\label{subsec:common_error}

The process of ``shifting the contour" in Perron's formula to prove a \emph{power saving} Tauberian theorem is a delicate process. Theorem \ref{thm:xi+1} requires not just a meromorphic continuation of the Dirichlet series, but also bounds in vertical strips. It is a common misconception to conflate the right half plane ${\rm Re}(s) > \sigma_a - \delta$ with power savings $O(X^{\sigma_a - \delta+\epsilon})$ for the error term.

We believe that Lee--Oh conflated these concepts in \cite{lee-oh2012} when studying $\theta(\Q,C_p)$. In particular, they make the following claim:
\begin{align*}    
    \text{\emph{If $g(s)$ is holomorphic on ${\rm Re}(s) > \sigma$ then $\displaystyle\frac{1}{2\pi i}\int_{c-i\infty}^{c+i\infty} g(s) \frac{X^s}{s}ds = O(X^{\sigma+\epsilon})$.}}
\end{align*}
Their meromorphic continuation results prove that the generating series has finitely many poles with real part $>\frac{1}{2a(C_p)}$, respectively $>\frac{1}{4a(C_p)}$ under GRH. The above claim would imply that $\theta(\Q,C_p) \le \frac{1}{2a(C_p)}$, respectively $\le \frac{1}{4a(C_p)}$ under GRH.

Lee--Oh cite Delange's Tauberian theorem for this (see \cite[Theorem III]{delange1954} or \cite[II Theorem 7.13 and 7.28]{tenenbaum2015}), however this result does not apply to just any holomorphic function. It only applies to Dirichlet series which are \emph{convergent} for ${\rm Re}(s) > \sigma$, not just holomorphic for ${\rm Re}(s) > \sigma$. Any error bound produced from a contour integral beyond the region of convergence (or for non-Dirichlet series) will necessarily incorporate the growth of $g(s)$ in vertical strips.

In particular, this claim is demonstrably false for $g(s) = \zeta(s)^2 - \frac{1}{(s-1)^2} - \frac{2\gamma}{s-1}$, where $\gamma$ the Euler--Mascheroni constant. The rational function pieces cancel out the poles of $\zeta(s)^2$, making $g(s)$ an entire function. Moreover,
\[
    \Delta(X) = \frac{1}{2\pi i}\int_{c-i\infty}^{c+i\infty} g(s) \frac{X^s}{s}ds
\]
is equal to the error term for Dirichlet's divisor problem
\[
    \Delta(X) = \sum_{n\le X} \left(\sum_{d\mid n} 1\right) - X\log X - (2\gamma - 1)X.
\]
The claim would suggest that $\Delta(X) = O(X^\epsilon)$, as $g(s)$ is holomorphic on the right halfplane ${\rm Re}(s) >0$. However, Hardy previously showed that $\inf\{\theta : \Delta(X) = O(X^\theta)\} \ge \frac{1}{4}$ in \cite{hardy1916}, contradicting the claim.

For this reason, we believe that the bounds for $\theta(\Q,C_p)$ claimed in \cite{lee-oh2012} do not follow from the arguments in that paper. To the author's knowledge, the bounds in this paper are the best known for $\theta(\Q,C_p)$ for $p\ge 3$. Using S\"ohne's subconvexity bounds, Corollary \ref{cor:uncond} implies
\[
    \theta(\Q,C_p) \le \frac{1}{a(C_p)}\left(\frac{p+2}{p+5}\right) = \frac{p+2}{(p-1)(p+5)},
\]
whereas Corollary \ref{cor:GLH} with GRH implies
\[
    \theta(\Q,C_p)\le \frac{1}{2a(C_p)} = \frac{1}{2(p-1)}.
\]
Unfortunately, neither of these bounds are small enough to reveal the secondary term $X^{\frac{1}{3(p-1)}}$ claimed by \cite{lee-oh2012} under GRH. The generating Dirichlet series does have a pole at $s=\frac{1}{3(s-1)}$ as described in \cite{lee-oh2012}, but, in light of \cite{hardy1916} and similar results, it is not clear if the error bounds can be made small enough to reveal the corresponding asymptotic term.

The power savings claimed in \cite{lee-oh2012} can be proven for \emph{smoothed} counting functions, such as
\[
\sum_{\Gal(K/\Q)\cong C_p} \exp\left(1-\frac{|\disc(K/\Q)|}{X}\right),
\]
in which case there would be an asymptotic term of size $X^{\frac{1}{3(p-1)}}$. This is because the rapid decay of the smoothing function $\exp(1-\frac{n}{X})$ overrides any polynomial growth of the generating series in vertical strips, so that the region of meromorphic continuation actually does match the smoother error bounds.

\section{Subconvexity}\label{sec:subconvexity}

Vertical bounds of the form
\[
|L(s)| \ll_L (1+|t|)^{\beta}
\]
are a topic of great interest in analytic number theory. We give a brief overview of this area, based on the material in \cite{iwaniec-kowalski2004} for the Selberg class of $L$-functions. Strictly speaking, this is called the \textbf{$t$-aspect} of vertical bounds, as it relates to the growth in $t$. There is also the \textbf{$q$-aspect}, which asks how the implicit constant depends on the conductor $q(L)$ of the $L$-function
\[
    |L(s)|\ll q(L)^\alpha (1+|t|)^{\beta}.
\]
We restrict ourselves to considering only the $t$-aspect, as this is what we require to prove the main results of this paper.

Most objects that one might call an ``$L$-function" (such as those functions in the Selberg class) have a function $\mu_L:\R\to \R_{\ge 0}$ associated to them defined as
\[
\mu_L(\sigma) = \inf\{\xi : |L(\sigma+it)|\ll_L (1+|t|)^{\xi}\text{ for all }|t|\text{ sufficiently large}\}.
\]
The Phragmen-Lindel\"of principle implies that $\mu_L$ is a \textbf{convex function}. If $\sigma_a$ is the abscissa of absolute convergence for $L(s)$, it is clear that $\mu_L(\sigma) = 0$ for $\sigma > \sigma_a$. If $L(s)$ has a functional equation, this symmetry can be used to determine $\mu_L(\sigma)$ for $\sigma < 0$. For example, if $L(s)$ belongs to the Selberg class with degree $d$, then $\mu_L(\sigma) = \frac{d}{2} - \sigma$ for $\sigma < 0$.

The behavior of $\mu_L(\sigma)$ in the critical strip $(0,\sigma_a)$ is a subject of great interest, in part because of the role these bounds play in Tauberian theorems like Theorem \ref{thm:xi+1}.

\begin{proposition}
    Let $L(s)$ be an $L$-function in the Selberg class of degree $d$, so that the abscissa of absolute convergence is $\sigma_a=1$. Then
    \[
    \mu_L(\sigma) \le \frac{d}{2}(1-\sigma)
    \]
    in the critical strip $0\le \sigma \le 1$. This is called the \textbf{convexity bound}.
\end{proposition}

This proposition follows from the fact that $\mu_L(\sigma)$ is convex, and interpolating between the regions $\sigma < 0$ and $\sigma > 1$.

Any bound that beats the convexity bound is called a \textbf{subconvexity bound}. These are often presented as bounds specifically for $\mu_L(1/2)$, from which convexity of $\mu_L$ gives a corresponding improved bound for all $\sigma$ in the critical strip. Proving new subconvexity bounds is a very rich, very active area of number theory. It would be unreasonable to attempt to summarize the best known results for all $L$-functions, especially as results for individual $L$-functions are constantly being improved. We will restrict ourselves to a general, uncomplicated subconvexity bound for Hecke $L$-functions:

\begin{theorem}[{S\"ohne \cite{sohne1997}}]
    Let $L(s,\chi)$ be a Hecke $L$-function over the number field $K$. Then $L(s,\chi)$ is an $L$-function of degree $[K:\Q]$ and
    \[
    \mu_\chi(1/2) \le \frac{[K:\Q]}{6}.
    \]
    This beats the convexity bound of $\mu_\chi(1/2)\le \frac{[K:\Q]}{4}$.
\end{theorem}

The most one could hope for is that $\mu_L(1/2)=0$. This is the content of the generalized Lindel\"of Hypothesis:

\begin{conjecture}[Generalized Lindel\"of Hypothesis]\label{conj:glh}
    Let $L(s)$ be an $L$-function in the Selberg class. Then $\mu_L(1/2) = 0$.
\end{conjecture}

Conjecture \ref{conj:glh} would follow from knowing that $L(s)$ satisfies GRH. In this case, convexity of $\mu_L$ implies that
\[
\mu_L(\sigma) = d\max\left\{\frac{1}{2}-\sigma,0\right\}
\]
where $d$ is the degree of $L(s)$.

In our setting, the vertical bounds for $\widehat{w_s}(h)$ are built from the vertical bounds for several Artin $L$-functions along with the fact that $B(s)$ is absolutely convergent on the region ${\rm Re}(s) \ge 1/2a(T)$.

\begin{corollary}\label{cor:w_vertical_bounds}
    Let $w_s:H^1(\A_K,T(\pi))\to \C$ be given by $w_s(f) = |\disc(f*\pi)|^{-s}$ for each complex $s\in \C$ and $h\in H^1(K_p,T(\pi)^*)$. Let $\mu_h(\sigma)$ be the vertical bounds associated to the Dirichlet series $\widehat{w_s}(h)$, that is
    \[
    \mu_h(\sigma) = \inf\{\xi : |\widehat{w_s}(h)| \ll (1+|t|)^{\xi}\text{ for all }s=\sigma+it\text{ with }|t|\text{ sufficiently large}\}.
    \]
    Then
    \begin{align*}
        \mu_h(\sigma) \le \sum_{\substack{o\subseteq \Hom((\Q/\Z)^*,T(\pi))\\o\ne \{1\}}} 2\mu_{\rho_h|_{\C o}}(1/2)\max\{1-\ind(o)\sigma,0\}
    \end{align*}
    for each $\sigma > \frac{1}{2a(T)}$, where $\rho_h$ is the twisted permutation representation given in Definition \ref{def:twisted_perm_rep} and $\mu_\rho$ is the vertical bound function associated to the Artin $L$-function $L(s,\rho)$.
\end{corollary}

\begin{proof}
    The proof is immediate from the vertical bounds of each factor in Corollary \ref{cor:meromorphic_continuation}.
\end{proof}

Given that $\rho_h|_{\C o}$ is induced from a Hecke character, Corollary \ref{cor:w_vertical_bounds} gives vertical bounds for $\widehat{w_s}(h)$ in terms of vertical bounds for Hecke $L$-functions. Together with Theorem \ref{thm:xi+1}, this implies that better subconvexity bounds for Hecke $L$-functions translate to better power savings for counting abelian extensions.

\section{Optimizing the Power Savings}\label{sec:optimizing}

We are now ready to state and prove the explicit bounds for $\theta(K,G)$ and $\theta(K,T(\pi))$ in terms of subconvexity bounds for Hecke $L$-functions.

\begin{theorem}\label{thm:best_bound}
    Let $K$ be a number field, $G\subseteq S_n$ a transitive subgroup, $T\normal G$ abelian, and $\pi:G_K\to G$ a continuous homomorphism. For each real $a(T) \le D \le 2a(T)$,
    \[
    \theta(K,T(\pi)) \le \frac{1}{a(T)} - \frac{\displaystyle\frac{1}{a(T)} - \frac{1}{D}}{\displaystyle 1+\underset{h\in H_{ur^{\perp}}^1(K,T(\pi)^*)}{\max}\left\{\sum_{\substack{o\subseteq \Hom((\Q/\Z)^*,T(\pi))\\0 < \ind(o) < D}} 2\mu_{\rho_h|_{\C o}}(1/2)\left(1-\frac{\ind(o)}{D}\right)\right\}},
    \]
    where $\rho_h$ is the twisted permutation representation given by Definition \ref{def:twisted_perm_rep}.
    
    Moreover, if $f(D)$ is the function giving this upper bound then $f(D)$ is a convex function on the interval $[a(T),2a(T)]$ whose value minimum value occurs at one of $D\in \ind(T-\{1\})$ or $D=2a(T)$.
\end{theorem}

The fact that $\rho_h$ is a twisted permutation representation means that $\rho_h|_{\C o}$ is induced from a Hecke character, so that these bounds are determined solely by subconvexity bounds of Hecke $L$-functions. The fact that the upper bound is given by a convex function means that it has a unique \emph{local} minimum, which helped us write more efficient code for Data Analysis \ref{data}.

The bounds given in the introduction immediately follow from Theorem \ref{thm:best_bound}. Additionally, the second statement implies that these bounds are optimized - one of the finitely many values of $D$ listed will give the smallest power saving bound produced by Theorem \ref{thm:xi+1}. Precisely which value of $D$ achieves this minimum will depend on the subconvexity bounds used.

\begin{proof}[Proof of Corollary \ref{cor:uncond} and Corollary \ref{cor:uncond_twist}]
    The representation $\rho_h|_{\C o}$ is induced from a Hecke character on $\overline{K}^{\Stab_{\overline{\rho}_h}(o))}$. Notice that Definition \ref{def:twisted_perm_rep} implies that the true permutation representation $\overline{\rho}_h$ is independent of $h$, and is the permutation representation given by the Galois action on $\Hom((\Q/\Z)^*,T(\pi))$. Thus, $\Stab_{\overline{\rho}_h}(o) = \Stab(o)$ is the Galois stabilizer of $o\subseteq \Hom((\Q/\Z)^*,T(\pi))$. This implies $L(s,\rho_h|_{\C o})$ is an $L$-function of degree
    \[
        [\overline{K}^{\overline{\rho}_h^{-1}(\Stab(o))}:\Q] = [G_K:\Stab(o)][K:\Q].
    \]
    S\"ohne's subconvexity bound for Hecke $L$-functions \cite{sohne1997} implies that
    \[
        \mu_{\rho_h|_{\C o}}(1/2) \le \frac{[G_K:\Stab(o)][K:\Q]}{6}
    \]
    for each $h\in H_{ur}^1(K,T(\pi)^*)$ and each orbit $o\subseteq \Hom((\Q/\Z)^*,T(\pi))$. We remark that $\ind(o) = 0$ if and only if $o=\{1\}$, in which case $L(\ind(0)s,\rho_h|_{\C o}) = \zeta(0)$ is a nonzero constant. Plugging this bound into Theorem \ref{thm:best_bound} gives
    \begin{align*}
        \theta(K,T(\pi)) &\le \frac{1}{a(T)} - \frac{\displaystyle\frac{1}{a(T)} - \frac{1}{D}}{\displaystyle 1+\sum_{\substack{o\subseteq \Hom((\Q/\Z)^*,T(\pi))\\0 < \ind(o) < D}} \frac{[G_K:\Stab(o)][K:\Q]}{3}\left(1-\frac{\ind(o)}{D}\right)}.
    \end{align*}
    Orbit--stabilizer implies that
    \begin{align*}
        \theta(K,T(\pi)) &\le \frac{1}{a(T)} - \frac{\displaystyle\frac{1}{a(T)} - \frac{1}{D}}{\displaystyle 1+\sum_{\substack{o\subseteq \Hom((\Q/\Z)^*,T(\pi))\\0 < \ind(o) < D}} \frac{|o|[K:\Q]}{3}\left(1-\frac{\ind(o)}{D}\right)}\\
        &=\frac{1}{a(T)} - \frac{\displaystyle\frac{1}{a(T)} - \frac{1}{D}}{\displaystyle 1+\sum_{\substack{\alpha\in \Hom((\Q/\Z)^*,T(\pi))\\0 < \ind(\alpha) < D}} \frac{[K:\Q]}{3}\left(1-\frac{\ind(\alpha)}{D}\right)}.
    \end{align*}
    Choose an isomorphism $\Hom((\Q/\Z)^*,T(\pi))\cong T$ as groups which preserves the index of elements. This implies
    \begin{align*}
        \theta(K,T(\pi)) &\le \frac{1}{a(T)} - \frac{\displaystyle\frac{1}{a(T)} - \frac{1}{D}}{\displaystyle 1+\sum_{\substack{g\in T-\{1\}\\\ind(g) < D}} \frac{[K:\Q]}{3}\left(1-\frac{\ind(g)}{D}\right)},
    \end{align*}
    proving Corollary \ref{cor:uncond_twist}. Corollary \ref{cor:uncond} follows by taking $G=T$ and $\pi=1$, as $\theta(K,G) = \theta(K,G(1))$.
\end{proof}

\begin{proof}[Proof of Corollary \ref{cor:GLH} and Corollary \ref{cor:GLH_twist}]
    The generalized Lindel\"of Hypothesis implies that $\mu_{\rho_h|_{\C o}}(1/2) = 0$. Thus, taking $D=2a(T)$ in Theorem \ref{thm:best_bound} implies
    \[
        \theta(K,T(\pi)) \le \frac{1}{a(T)} - \frac{\displaystyle\frac{1}{a(T)} - \frac{1}{2a(T)}}{1+0} = \frac{1}{2a(T)}.
    \]
    This proves Corollary \ref{cor:GLH_twist}. Corollary \ref{cor:GLH} follows by taking $G=T$ and $\pi=1$.
\end{proof}

Theorem \ref{thm:best_bound} will follow from Theorem \ref{thm:xi+1} applied to $\widehat{w_s}(h)$ for each $h\in H^1(K,T(\pi))$ together with some Calculus to prove that the upper bound is convex.

\begin{proof}[Proof of Theorem \ref{thm:best_bound}]
    By definition, the absolute series for $\widehat{w_s}(h)$ is
    \[
    \frac{1}{|H^0(K,T(\pi))|}\prod_p \left(\sum_{f\in H^1(G_{K_p},T(\pi))} |\langle f, h|_{G_{K_p}}\rangle| |\disc(f*\pi)|^{-s}\right) = \widehat{w_s}(0)
    \]
    as the local Tate pairing is valued in the unit circle. Thus, for each $L(s) = \widehat{w_s}(h)$ we can take $\widehat{L}(s) = \widehat{w_s}(0)$. We know that $\widehat{\sigma}_a = 1/a(T)$ for this choice.
    
    We also know by Theorem \ref{thm:generating_mero_twist} that $\widehat{w_s}(h)$ is meromorphic on the right halfplane ${\rm Re}(s)>1/2a(T)$ with poles at $s=1/\ind(o)$. Thus, for each $\delta \in [0,1/2a(T)]$ and $\xi_h=\xi_h(\delta)$ for which
    \[
    |\widehat{w_s}(h)|,|\widehat{w_s}(0)| \ll (1+|t|)^{\xi_h+\epsilon}
    \]
    for all ${\rm Re}(s) > 1/a(G) - \delta$, Theorem \ref{thm:xi+1} implies that the sum of coefficients of $\widehat{w_s}(h)$ grows as
    \[
    \sum_{\substack{o\subseteq \Hom((\Q/\Z)^*,T(\pi))\\o\ne \{1\}}} \underset{s=1/\ind(o)}{\rm Res}\left(\widehat{w_s}(h)\frac{X^s}{s}\right) + O\left(X^{\frac{1}{a(G)} + \frac{\delta}{\xi_h+1}+\epsilon}\right).
    \]
    Taking a sum over $h\in H^1_{ur^{\perp}}(K,T(\pi)^*)$, this implies
    \begin{align*}
    \#H^1(K,T,\pi;X) =& \frac{|H^0(K,T(\pi))|}{|H^0(K,T(\pi)^*)|}\sum_{h\in H^1_{ur^{\perp}}(K,T(\pi)^*)}\sum_{\substack{o\subseteq \Hom((\Q/\Z)^*,T(\pi))\\o\ne \{1\}}} \underset{s=1/\ind(o)}{\rm Res}\left(\widehat{w_s}(h)\frac{X^s}{s}\right)\\
    &+ O\left(X^{\frac{1}{a(G)} + \frac{\delta}{\xi+1}+\epsilon}\right)\\
    =&\sum_{\substack{o\subseteq \Hom((\Q/\Z)^*,T(\pi))\\o\ne \{1\}}} \underset{s=1/\ind(o)}{\rm Res}\left(\frac{|H^0(K,T(\pi))|}{|H^0(K,T(\pi)^*)|}\sum_{h\in H^1_{ur^{\perp}}(K,T(\pi)^*)}\widehat{w_s}(h)\frac{X^s}{s}\right)\\
    &+ O\left(X^{\frac{1}{a(G)} + \frac{\delta}{\xi+1}+\epsilon}\right)
    \end{align*}
    where $\xi = \max_h \xi_h$.

    Sieving to surjective maps will sieve the Dirichlet series inside the residue to $D_{K,T(\pi)}(s)$. Note that if $L(s)$ is the generating series of crossed homomorphisms with $(f*\pi)(G_K) \subseteq H\subsetneq G$, then all that happens to the generating series is that it is missing terms. We can still take $\widehat{L}(s) = \widehat{w_s}(0)$ to bound the absolute series. This bound on the absolute series will be worse when $H$ is small, but this way we can use Theorem \ref{thm:xi+1} with the same value for $\widehat{\sigma}_a$. The poles of this generating series are in the same place of the same or smaller order as well, as restricting to the subgroup $H\subsetneq G$ will only modify Corollary \ref{cor:meromorphic_continuation} by removing any $L$-function factors for which $o\not\subseteq \Hom((\Q/\Z)^*,H(\pi))$. This will only make the vertical bounds smaller than $\xi_h$, so we can also use Theorem \ref{thm:xi+1} with the same value for $\xi_h$. The surjective sieve is then an alternating sum of asymptotic terms of the form given in Definition \ref{def:theta_twist} and error terms which are all $O(X^{\frac{1}{a(T)} - \frac{\delta}{\xi+1}+\epsilon})$. Thus, we have proven that
    \begin{align*}
    \#\Surj(G_K,T,\pi;X) &=\sum_{\substack{o\subseteq \Hom((\Q/\Z)^*,T(\pi))\\o\ne \{1\}}} \underset{s=1/\ind(o)}{\rm Res}\left(D_{K,T(\pi)}(s)\frac{X^s}{s}\right) + O\left(X^{\frac{1}{a(G)} + \frac{\delta}{\xi+1}+\epsilon}\right)
    \end{align*}

    For each $\delta\in [0,1/2a(T)]$, Corollary \ref{cor:w_vertical_bounds} implies that we can take
    \begin{align*}
        \xi = \xi(\delta) = \underset{h\in H_{ur^{\perp}}^1(K,T(\pi)^*)}{\max}\sum_{\substack{o\subseteq \Hom((\Q/\Z)^*,T(\pi))\\o\ne \{1\}}} 2\mu_{\rho_h|_{\C o}}(1/2)\max\left\{1-\frac{\ind(o)}{a(T)} + \ind(o)\delta,0\right\}.
    \end{align*}
    It is clear from this description that $\xi(\delta)$ is an increasing, continuous, piecewise linear function.

    Set $\theta(\delta) = \frac{1}{a(T)} - \frac{\delta}{\xi(\delta)+1}$. The first part of the Theorem follows immediately from setting $\delta = \frac{1}{a(T)} - \frac{1}{D}$ for each real $a(T)\le D < 2a(T)$ and noting that $\theta(K,T(\pi))\le \theta(\delta)$. For the case $D=2a(T)$, this follows from taking $\delta$ sufficiently close to $\frac{1}{a(T)} - \frac{1}{D}$ in terms of $\epsilon$ so that $\theta(K,T(\pi)) < \theta(\delta)+\epsilon$. It now remains to optimize $\theta(\delta)$ on the interval $a(T) \le D \le 2a(T)$ to prove the second part of the Theorem.
    
    $\theta(\delta)$ is continuous and piecewise differentiable, where the endpoints of the pieces are $0$, $\frac{1}{a(T)} - \frac{1}{\ind(o)}$ for each $o\subseteq \Hom((\Q/\Z)^*,T(\pi))$, and $\frac{1}{2a(T)}$.

    Restrict to the piece $\delta \in [1/a(T) - 1/n_1,1/a(T) - 1/n_2]$ for some $n_1>n_2$ coming from the set $\ind(T-\{1\}) \cup\{2a(T)\}$. On this piece,
    \[
    \theta(\delta) = \frac{1}{a(T)} - \frac{\delta}{m\delta+b}
    \]
    where
    \[
    b = b(n_1) = 1 + \sum_{\substack{o\subseteq \Hom((\Q/\Z)^*,T(\pi))\\0 < \ind(o) < n_1}} 2\mu_{\rho_h|_{\C o}}(1/2)\left(1-\frac{\ind(o)}{a(T)}\right).
    \]
    Differentiating on this piece gives
    \[
    \theta'(\delta) = \frac{-b(n_1)}{(m\delta+b(n_1))^2}.
    \]
    Thus, $\theta(\delta)$ is monotonic and in particular nondecreasing on this piece if and only if $b(n_1) \le 0$. It is clear that $b(n_1)$ is a decreasing function, as $\ind(o) \ge a(T)$ for all orbits $o$ so adding more terms to the sum amounts to adding extra negative numbers. Thus, $\theta'(\delta)$ is increasing in $n_1$.

    Given that $\frac{1}{a(T)} - \frac{1}{n_1}$ is increasing in $n_1$
    This proves that $\theta'(\delta)$ is increasing where it exists so that $\theta(\delta)$ is convex. The minimum value occurs at the endpoint of one of the pieces, namely an endpoint where $\theta'(\delta)$ switches sign. Thus, $f(D) = \theta(\frac{1}{a(G)} - \frac{1}{D})$ is a convex function whose minimum value occurs at one of $D\in \ind(T-\{1\})$ or $D=2a(T)$.
\end{proof}

\section{Obstructions to the Non-vanishing of Lower Order Terms}\label{sec:obstructions_lower_order}

We discuss the lower order terms in this section. This is primarily a discussion of the obstacles towards proving the lower order terms do not vanish, or equivalently proving that the poles of $D_{K,G}(s)$ in Theorem \ref{thm:generating_mero} do not vanish.

There are essentially three type of obstructions.
\begin{enumerate}[(1)]
\item Real zeros of $B(s)$ and of the Hecke $L$-functions that appear in Corollary \ref{cor:meromorphic_continuation} can reduce the order of the poles in Corollary \ref{cor:poles}. This is evident from the structure of Corollary \ref{cor:meromorphic_continuation}.

\item If $H_{ur^{\perp}}^1(K,T(\pi)^*)\ne 1$, then there may be cancellation between the poles of $\widehat{w_s}(h_1)$ and the poles of $\widehat{w_s}(h_2)$. This is tricky, as the residues of $\widehat{w_s}(h_j)\frac{X^s}{s}$ giving the lower order terms include special values of $L$-functions in the critical strip. The sum over $h$ of the coefficients in a potential lower order term $X^{1/d}P_{K,T(\pi),1/d}(\log X)$ will then become a linear combination of rational special values of an $L$-function times a convergent Euler product. Proving non-vanishing of linear combinations of $L$-functions at values in the critical strip is immensely difficult to do directly.

\item Cancellation between the poles during the sieve to surjective coclasses. For any $g\in T-\{1\}$ which does not generate $T$, the (potential) lower order term associated to the pole $s=1/\ind(g)$ will appear both in the term for all coclasses and in at least one other term of the sieve. This is an alternating sieve, and so becomes an alternating sieve of lower order terms that might cancel.
\end{enumerate}

In principle, these obstacles are all present for the main term with the exception that the main term does not involve any special values of $L$-functions in the critical strip. Non-cancellation of the main term is proven in a round-about way by giving a nonzero lower bound. By choosing restricted local conditions at finitely many places that avoid all three obstructions, one proves a lower bound of the right order of magnitude.

The same lower bound does not work for lower order terms. For example, suppose we knew that
\[
a_1X + b_1 X^{1/2} + O(X^{1/3}) \gg a_2 X + b_2 X^{1/2} + O(x^{1/3})
\]
for $a_2,b_2>0$. This does implies that $a_1\ge a_2 > 0$, but it does not necessarily say anything about $b_1$. Indeed, $2X \gg X+X^{1/2} + O(X^{1/3})$. When restricting local conditions at finitely many places, the leading coefficient necessarily changes which prevents us from making conclusions about the lower order terms.

It is not immediately clear to us that these obstructions can be overcome. The presence of a linear combination of special values of $L$-functions in the critical strip in particular gives us pause.

However, we can say more if we restrict to a special case. Suppose we know that $H_{ur^{\perp}}^1(K,T(\pi)^*) = 1$. This avoids the second obstruction entirely, and now the generating series of for $H^1(K,T(\pi))$ is given only by
\[
\frac{|H^0(K,T(\pi))|}{|H^0(K,T(\pi)^*)|}\widehat{w_s}(0),
\]
which is a single Euler product. Corollary \ref{cor:poles} gives the precise order of the poles for $\widehat{w_s}(0)$ in terms of the zeros of $L$-functions rather than an upper bound, also avoiding the first obstruction.

We can classify exactly when this holds for $G$ having the trivial Galois action.
\begin{lemma}\label{lem:QG}
    Let $G$ be an abelian group with the trivial Galois action. Then $H_{ur^{\perp}}^1(K,G^*)=1$ if and only if one of the following holds:
    \begin{enumerate}[(a)]
        \item $K=\Q$, or
        \item $G=C_2^n$ and $K$ has class number $1$.
    \end{enumerate}
\end{lemma}

In these cases, we completely avoid the first two obstructions and only need to worry about the sieve to surjective coclasses (which in this case that $G$ has the trivial Galois action would be the sieve from \'etale algebras to field extensions). Our explicit work in the following subsection suggests that the sieve to surjective coclasses is unlikely to cancel out the poles of $\widehat{w_s}(0)$. These facts are what we consider to be convincing evidence for Conjecture \ref{conj:Q_order_of_poles}.

\begin{proof}[Proof of Lemma \ref{lem:QG}]
    This follows from the Greenberg--Wiles identity \cite{wiles1995}
    \[
    \frac{|H^1_{\mathcal{L}}(K,G)|}{|H^1_{\mathcal{L}^{\perp}}(K,G^*)|} = \frac{|H^0(K,G)|}{|H^0(K,G^*)|}\prod_p \frac{|L_p|}{|H^0(K_p,G)|}.
    \]
    Letting $\mathcal{L} = (L_p)$ be the family $L_p = H_{ur}^1(K_p,G)$ at all finite $p$ and $L_p = H^1(K_p,G)$ at all infinite places.

    The Selmer group is given by
    \[
    H_{\mathcal{L}}(\Q,G) = H_{ur}^1(K,G) = \Hom({\rm Cl}(K),G)
    \]
    for ${\rm Cl}(K)$ the narrow class group of $K$. The dual Selmer group is precisely $H_{ur^{\perp}}^1(K,G^*)$ by construction.

    We know that $H^0(K,G) = |G|$ as $G$ has the trivial action. $H^0(K,G^*)=[G:|\mu(K)|G]$ as the only fixed points of $G^*=\Hom(G,\mu)$ are those $\alpha$ which land in $\mu\cap K = \mu(K)$.

    At the finite places, $|L_p|=|H_{ur}^1(K_p,G)| = |G| = |H^0(K_p,G)|$. At infinite places,
    \begin{align*}
        |L_p|=\begin{cases}
            1 & p\text{ complex}\\
            |G[2]| & p\text{ real}.
        \end{cases}
    \end{align*}

    Plugging these all in implies
    \begin{align*}
        \frac{|\Hom({\rm Cl}(K),G)|}{|H^1_{ur^{\perp}}(K,G^*)|} = \frac{|G|}{[G:|\mu(K)|G]}\frac{|G[2]|^{r_1(K)}}{|G|^{r_1(K) + r_2(K)}}
    \end{align*}
    for $r_1(K)$ and $r_2(K)$ the number of real (resp. complex) places of $K$. This implies
    \begin{align*}
        |H^1_{ur^{\perp}}(K,G^*)| &= |G|^{r_1(K) + r_2(K) - 1}\frac{[G:|\mu(K)|G]}{|G[2]|^{r_1(K)}}|\Hom({\rm Cl}(K),G)|.
    \end{align*}
    If $K=\Q$, then ${\rm Cl}(\Q) = 1$, $r_1(\Q) = 1$, $r_2(\Q) = 0$, and $|\mu(\Q)|=2$. Given that $|G[2]| = [G:2G]$, we conclude that $|H^1_{ur^{\perp}}(\Q,G^*)|=1$.

    If $K\ne \Q$, then either $r_2(K) \ge 1$ or $r_1(K)\ge 2$. If $r_2(K) \ge 1$ then
    \begin{align*}
        |H^1_{ur^{\perp}}(K,G^*)| &\ge [G:G[2]]^{r_1(K)}[G:|\mu(K)|G]|\Hom({\rm Cl}(K),G)| > 1,
    \end{align*}
    where the final inequality follows from knowing that $2\mid |\mu(K)|$. If $r_2(K) = 0$ and $r_1(K) \ge 2$, then
    \begin{align*}
        |H^1_{ur^{\perp}}(K,G^*)| &= |G|^{-1}[G:G[2]]^{r_1(K)}[G:|\mu(K)|G]|\Hom({\rm Cl}(K),G)|\\
        &= [G:G[2]]^{r_1(K)-1}[2G:|\mu(K)|G]|\Hom({\rm Cl}(K),G)|.
    \end{align*}
    This is strictly larger than $1$ unless $G=G[2]$ and $K$ has trivial class group.
\end{proof}

\section{Some Examples of Non-vanishing of Lower Order Terms}\label{sec:examples_lower_order}

When $K=\Q$, there are some tricks we can employ to avoid or greatly simplify the obstructions in Section \ref{sec:obstructions_lower_order}. This is how we prove Theorem \ref{thm:nonvanishing} and Corollary \ref{cor:C4andC6}.

\subsection{An Inclusion-Exclusion Lemma for Sieving Euler Products}

We first prove a lemma for an inclusion-exclusion sieve of convergent Euler products. This will be applied to show that the sieve to surjective coclasses does not cancel out certain poles, as long as the sieve does not involve too many special values of $L$-functions in the critical strip.

\begin{lemma}\label{lem:poset_nonvanishing}
Let $(P,\le)$ be a poset with M\"obius function $\mu_P$ in which all joins exist, and $f:P\times \N\to \R_{\ge 0}$ a nonegative function for which $f(x,n)$ is multiplicative in $n$ and supported on the squarefree integers for each $x\in P$.

Let $x\in P$ such that $\#\{y\le x : \mu_P(y,x)\ne 0\}$ is finite and suppose that
\[
\prod_p\left(1+\sum_{y\le x} f(y,p)\right)
\]
converges. Then
\[
\sum_{z\le x} \mu_P(z,x) \prod_p\left(1+\sum_{y\le z} f(y,p)\right) = \left(\sum_{z\le x} \mu_P(z,x)\right) + \sum_{n=2}^\infty \mu(n)^2 \sum_{\substack{(y_p)_{p\mid n}\\\vee_{p\mid n} y_p = x}} \prod_{p\mid n} f(y_p,p).
\]
In particular, this is positive if and only if $x$ is minimal or there exists at least one term $\prod_{p\mid n} f(y_p,p) > 0$ with $\vee_{p\mid n} y_p = x$.
\end{lemma}

\begin{proof}
    All of the Euler products in the statement are absolutely convergent as $z\le x$ implies
    \[
    \prod_p\left(1+\sum_{y\le z} f(y,p)\right) \le \prod_p\left(1+\sum_{y\le x} f(x,p)\right) < \infty.
    \]
    First suppose $x\in P$ is minimal, that is $\#\{y\in P : y< x\} = 0$. Then
    \begin{align*}
        \sum_{z\le x} \mu_P(z,x) \prod_p\left(1+\sum_{y\le z} f(y,p)\right) &= \prod_p(1 + f(x,p))\\
        &= \sum_{n=1}^{\infty} \mu(n)^2 f(x,p)\\
        &= \left(\sum_{z\le x}\mu_P(z,x)\right) + \sum_{n=2}^{\infty} \mu(n)^2 f(x,p).
    \end{align*}

    For the remainder of the proof, suppose that $x\in P$ is not minimal. We distribute the Euler products to write
    \begingroup
    \allowdisplaybreaks
    \begin{align*}
        \sum_{z\le x} \mu_P(z,x) \prod_p\left(1+\sum_{y\le z} f(y,p)\right) &= \sum_{z\le x} \mu_P(z,x) \sum_{n=1}^{\infty} \sum_{\substack{(y_p)_{p\mid n}\\y_p\le z}} \mu(n)^2\prod_{p\mid n} f(y_p,p)\\
        &=\sum_{n=1}^{\infty}\mu(n)^2 \sum_{z\le x} \sum_{\substack{(y_p)_{p\mid n}\\y_p\le z}}\mu_P(z,x)\prod_{p\mid n} f(y_p,p)\\
        &=\sum_{n=1}^{\infty}\mu(n)^2 \sum_{\substack{(y_p)_{p\mid n}\\y_p\le x}}\left(\sum_{\substack{z\ge y_p\\\text{for each }p\mid n}} \mu_P(z,x)\right)\prod_{p\mid n} f(y_p,p)\\
        &=\sum_{n=1}^{\infty}\mu(n)^2 \sum_{\substack{(y_p)_{p\mid n}\\y_p\le x}}\left(\sum_{z\ge \vee_{p\mid n} y_p} \mu_P(z,x)\right)\prod_{p\mid n} f(y_p,p).
    \end{align*}
    \endgroup
    The M\"obius function of a poset satisfies
    \[
    \sum_{z\ge y}\mu_P(z,x) = \begin{cases}
        1 & y=x\\
        0 & y\ne x.
    \end{cases}
    \]
    Thus, the expression simplifies to
    \begin{align*}
        \sum_{z\le x} \mu_P(z,x) \prod_p\left(\sum_{y\le z} f(y,p)\right) &=\sum_{n=2}^{\infty}\mu(n)^2 \sum_{\substack{(y_p)_{p\mid n}\\\vee_{p\mid n} y_p=x}}\prod_{p\mid n} f(y_p,p)
    \end{align*}
    Note that that starting index changes to a $2$ because when $n=1$, $\vee_{p\mid n}y_p$ is an empty join. Thus, the sum is over all $z$ or equivalently over all finitely many $z\le x$ with $\mu_P(z,x)\ne 0$. As $x$ is not minimal, this term cancels out.
\end{proof}

\subsection{Proving Theorem \ref{thm:nonvanishing}}

Lemma \ref{lem:QG} implies that $H_{ur}^1(\Q,G^*)=1$, so that the generating series for $G$-\'etale algebras order by dscriminant is
\[
\frac{|H^0(\Q,G)|}{|H^0(\Q,G^*)|}\widehat{w_s}(0) = \prod_{p<\infty} \left(\frac{1}{|G|}\sum_{f_p\in \Hom(G_{\Q_p},G)} |\disc(f_p)|^{-s}\right).
\]
This factors according to Corollary \ref{cor:meromorphic_continuation} as
\[
B_G(s)\prod_{o\subseteq \Hom((\Q/\Z)^*,G)} \zeta_{\overline{\Q}^{{\rm Stab}(o)}}(\ind(o)s),
\]
where $B_G(s)$ is an absolutely convergent Euler product on the region ${\rm Re}(s) > 1/2a(G)$. This is because $\rho_0=\overline{\rho}_0$ in Definition \ref{def:twisted_perm_rep} is the true permutation representation associated to the Galois action on $\Hom((\Q/\Z)^*,G)$. Notice that if $g\in o$ has order $m$, then $\overline{\Q}^{{\rm Stab}(o)} = \Q(\zeta_m)$. Thus, Corollary \ref{cor:poles} implies this series has poles at $s=1/d$ for each $d\in \ind(G-\{1\})$ of order exactly $\bar{b}_d(\Q,G)$.

The sieve to $G$-extensions from \'etale algebras is precisely the sieve to surjective homomorphisms $\Hom(G_\Q,G)$ to $\Surj(G_\Q,G)$. Let $\mu_G$ be the M\"obius function for the poset of subgroups of $G$, then
\[
D_{\Q,G}(s) = \sum_{H\le G} \mu_G(H,G) B_H(s)\prod_{g\in H-\{1\}} \zeta_{\Q(\zeta_{|\langle g\rangle|})}(\ind(o)s)^{1/[\Q(\zeta_{|\langle g\rangle|}):\Q]}.
\]
Conjecture \ref{conj:Q_order_of_poles} follows from this expression. The M\"obius function $\mu_G(H,G)$ is supported on $H$ contianing the Frattini subgroup. The summand corresponding to $H$ has a pole at $s=1/d$ of order exactly $\bar{b}_{d/[G:H]}(\Q,H)$. Conjecture \ref{conj:Q_order_of_poles} is a prediction that none of these poles cancel out in the sum.

Each part of Theorem \ref{thm:nonvanishing} is proven by analyzing this expression, and showing that the relevant pole of a single summand does not cancel out in the sieve.

We will actually prove Theorem \ref{thm:nonvanishing}(iv) first, as the proof is many times shorter.
\begin{proof}[Proof of Theorem \ref{thm:nonvanishing}(iv)]
The elements of $G=C_n$ which have index $n-1$ are precisely the generators. Thus, the only subgroup $H\le G$ that contain an element of index $n-1$ is $H=G$ itself. Thus, from the expression
\[
D_{\Q,C_n}(s) = \sum_{H\le C_n} \mu_{C_n}(H,C_n) B_H(s)\prod_{g\in H-\{1\}} \zeta_{\Q(\zeta_{|\langle g\rangle|})}(\ind(o)s)^{1/[\Q(\zeta_{|\langle g\rangle|}):\Q]}
\]
only the term associated to $H=C_n$ has a pole at $s=\frac{1}{n-1}$. This implies
\begin{align*}
&\lim_{s\to \frac{1}{n-1}}\left(s-\frac{1}{n-1}\right)^{\bar{b}_{n-1}(\Q,C_n)} D_{\Q,C_n}(s)\\
&= \lim_{s\to \frac{1}{n-1}}\left(s-\frac{1}{n-1}\right)^{\bar{b}_{n-1}(\Q,C_n)}B_{C_n}(s)\prod_{g\in C_n-\{1\}} \zeta_{\Q(\zeta_{|\langle g\rangle|})}(\ind(o)s)^{1/[\Q(\zeta_{|\langle g\rangle|}):\Q]}\\
&\ne 0
\end{align*}
by construction.
\end{proof}

This is a nifty trick, which has been used to show the leading term does not vanish when abelian extensions are ordered by the product of ramified primes. Essentially, only one term from the surjective sieve has a pole at the location of interest, so there is no chance of cancellation.

For the remaining parts, we leverage Lemma \ref{lem:poset_nonvanishing} to prove that the poles do not cancel. In parts (i,ii), \emph{every} term of the surjective sieve contributes the same zeta factors. By factoring these out, we are able to prove noncancellation of the pole.

\begin{proof}[Proof of Theorem \ref{thm:nonvanishing}(i,ii)]
We remark that part (i) actually follows from part (ii), as
\[
\{g\in G-\{1\} : \ind(g) < a(G)\} = \emptyset
\]
by definition, and $\emptyset \subseteq \Phi(G)$ is vacuously true.

The M\"obius function $\mu_G(H,G)$ for subgroups of $G$ is supported on subgroups containing the Frattini subgroup. Thus, we can factor several zeta functions out of the sum over $H\le G$. This proves that the product
\[
D_{\Q,G}(s)\prod_{\substack{g\in G-\{1\}\\\ind(g) < d}}\zeta_{\Q(\zeta_{|\langle g\rangle|})}(\ind(g)s)^{-\frac{b_{\ind(g)}(\Q,G)}{[\Q(\zeta_{|\langle g\rangle|}):\Q]}}
\]
is equal to
\begin{align*}
&\sum_{H\le G} \mu_G(H,G)\prod_{p<\infty}\left(\prod_{\substack{g\in G-\{1\}\\\ind(g) < d}}\left(1 - p^{-\ind(g)sf_p(\Q(\zeta_{|\langle g\rangle|})/\Q)}\right)^{\frac{b_{\ind(g)}(\Q,G)g_p(\Q(\zeta_{|\langle g\rangle|})/\Q)}{[\Q(\zeta_{|\langle g\rangle|}):\Q]}}\right)\\
&\times\left(\frac{1}{|H|}\sum_{f_p\in \Hom(G_{\Q_p},H)} |\disc(f_p)|^{-s}\right)\\
=&\sum_{H\le G} \mu_G(H,G)\prod_{p<\infty}\left(\prod_{\substack{g\in G-\{1\}\\\ind(g) < d}}\left(1 - p^{-\ind(g)sf_p(\Q(\zeta_{|\langle g\rangle|})/\Q)}\right)^{\frac{b_{\ind(g)}(\Q,G)g_p(\Q(\zeta_{|\langle g\rangle|})/\Q)}{[\Q(\zeta_{|\langle g\rangle|}):\Q]}}\right)\\
&\times\left(\frac{1}{|G|}\sum_{\substack{f_p\in \Hom(G_{\Q_p},G)\\f_p(I_p)\subseteq H}} |\disc(f_p)|^{-s}\right)
\end{align*}
Corollary \ref{cor:meromorphic_continuation} implies that these products are absolutely convergent for ${\rm Re}(s) \ge 1/d$. (In fact, for ${\rm Re}(s) > 1/m$ for $m\in \ind(G-\{1\})$ the smallest index with $d<m$).

Define the function $f:{\rm Sub}(G) \times \N \to \C$ so that $n\mapsto f(J,n)$ is multiplicative, $f(J,p^j)=0$ for each $j>1$, $f(1,p)=0$ for the trivial subgroup, and
\[
f(J,p) = \left(\prod_{\substack{g\in G-\{1\}\\\ind(g) < d}}\left(1 - p^{-\frac{\ind(g)f_p(\Q(\zeta_{|\langle g\rangle|})/\Q)}{d}}\right)^{\frac{b_{\ind(g)}(\Q,G)g_p(\Q(\zeta_{|\langle g\rangle|})/\Q)}{[\Q(\zeta_{|\langle g\rangle|}):\Q]}}\right)\left(\frac{1}{|G|}\sum_{\substack{f_p\in \Hom(G_{\Q_p},G)\\f_p(I_p) = J}} |\disc(f_p)|^{-1/d}\right)
\]
otherwise. This is chosen so that
\begin{align*}
\lim_{s\to 1/d}D_{\Q,G}(s)\prod_{\substack{g\in G-\{1\}\\\ind(g) < d}}\zeta_{\Q(\zeta_{|\langle g\rangle|})}(\ind(g)s)^{-\frac{b_{\ind(g)}(\Q,G)}{[\Q(\zeta_{|\langle g\rangle|}):\Q]}}&= \sum_{H\le G} \mu_G(H,G)\prod_{p<\infty} \left(1 + \sum_{J\le H} f(J,p)\right),
\end{align*}
where we already know the limit exists by Theorem \ref{thm:generating_mero}. The function $f$ satisfies the hypotheses of Lemma \ref{lem:poset_nonvanishing} by Corollary \ref{cor:meromorphic_continuation}. Thus, Lemma \ref{lem:poset_nonvanishing} implies that
\begin{align*}
\lim_{s\to 1/d} D_{\Q,G}(s)\prod_{\substack{g\in G-\{1\}\\\ind(g) < d}}\zeta_{\Q(\zeta_{|\langle g\rangle|})}(\ind(g)s)^{-\frac{b_{\ind(g)}(\Q,G)}{[\Q(\zeta_{|\langle g\rangle|}):\Q]}}&= \sum_{n=2}^{\infty}\mu(n)^2\sum_{\substack{(J_p)_{p\mid n}\\\prod J_p = G}}\prod_{p\mid n} f(J_p,p).
\end{align*}
To prove part (ii), it suffices to prove that this limit is nonzero. As this is a sum of nonegative terms, we can bound the limit below by a single nonzero term to complete the proof. Let $(p_g)_{g\in G-\{1\}}$ be a sequence of distinct primes congruent to $1\pmod {|\langle g\rangle|}$. Thus
\begin{align*}
    &\lim_{s\to 1/d} D_{\Q,G}(s)\prod_{\substack{g\in G-\{1\}\\\ind(g) < d}}\zeta_{\Q(\zeta_{|\langle g\rangle|})}(\ind(g)s)^{-\frac{b_{\ind(g)}(\Q,G)}{[\Q(\zeta_{|\langle g\rangle|}):\Q]}}\\
    &\ge \prod_{g\in G-\{1\}} f_\sigma(\langle g\rangle, p_g)\\
    &= \prod_{g\in G-\{1\}}\left(\prod_{\substack{h\in G-\{1\}\\\ind(h) < d}}\left(1 - {p_g}^{-\frac{\ind(h)}{d}}\right)^{b_{\ind(h)}(\Q,G)}\right)\left(\frac{1}{|G|}\sum_{\substack{f_{p_g}\in \Hom(G_{\Q_{p_g}},G)\\f_p(I_p) = \langle g\rangle}} |\disc(f_{p_g})|^{-1/d}\right)\\
    &= \prod_{g\in G-\{1\}}\left(\prod_{\substack{h\in G-\{1\}\\\ind(h) < d}}\left(1 - {p_g}^{-\frac{\ind(h)}{d}}\right)^{b_{\ind(h)}(\Q,G)}\right)p^{-\frac{\ind(g)}{d}}\\
    & > 0.
\end{align*}
This concludes the proof of part (ii), as we know that
\[
\prod_{\substack{g\in G-\{1\}\\\ind(g) < d}}\zeta_{\Q(\zeta_{|\langle g\rangle|})}(\ind(g)s)^{\frac{b_{\ind(g)}(\Q,G)}{[\Q(\zeta_{|\langle g\rangle|}):\Q]}}
\]
has a pole of order $\bar{b}_d(\Q,G)$ at $s=1/d$ by construction.
\end{proof}

For part (iii), we allow a slightly more complicated scenario - the is a $\zeta(a(G)s)$ factor that will appear on some terms of the sieve, and not on other terms. But otherwise, all Dedekind zeta functions that appear will appear on all terms of the sieve and can be factored out. We then take advantage of the fact that $\zeta(x) < 0$ for all real $x\in [0,1)$.

\begin{proof}[Proof of Theorem \ref{thm:nonvanishing}(iii)]
The start of the proof is similar to Theorem \ref{thm:nonvanishing}(ii), except now we know that for $G=C_2\times A$ each element $g\in G-\{1\}$ with $\ind(g) < d$ is \emph{either} (1) contained in the Frattini subgroup \emph{or} (2) of order $2$, so that $\ind(g) = a(G)$.

We factor out the zeta functions corresponding to $g$ with $a(G) < \ind(g) < d$ similarly to part (ii), so that
\[
D_{\Q,G}(s)\prod_{\substack{g\in G-\{1\}\\a(G) < \ind(g) < d}}\zeta_{\Q(\zeta_{|\langle g\rangle|})}(\ind(g)s)^{-\frac{b_{\ind(g)}(\Q,G)}{[\Q(\zeta_{|\langle g\rangle|}):\Q]}}
\]
is equal to
\begin{align*}
    \zeta(a(G)s)\sum_{\substack{H\le G\\C_2\subseteq H}}\mu_G(H,G)\prod_{p<\infty}&\left(\prod_{\substack{g\in G-\{1\}\\\ind(g) < d}}\left(1 - p^{-\ind(g)sf_p(\Q(\zeta_{|\langle g\rangle|})/\Q)}\right)^{\frac{b_{\ind(g)}(\Q,G)g_p(\Q(\zeta_{|\langle g\rangle|})/\Q)}{[\Q(\zeta_{|\langle g\rangle|}):\Q]}}\right)\\
    &\times\left(\frac{1}{|G|}\sum_{\substack{f_p\in \Hom(G_{\Q_p},G)\\f_p(I_p)\subseteq H}} |\disc(f_p)|^{-s}\right)\\
    + \sum_{\substack{H\le G\\C_2\not\subseteq H}}\mu_G(H,G)\prod_{p<\infty}&\left(\prod_{\substack{g\in G-\{1\}\\a(G)<\ind(g) < d}}\left(1 - p^{-\ind(g)sf_p(\Q(\zeta_{|\langle g\rangle|})/\Q)}\right)^{\frac{b_{\ind(g)}(\Q,G)g_p(\Q(\zeta_{|\langle g\rangle|})/\Q)}{[\Q(\zeta_{|\langle g\rangle|}):\Q]}}\right)\\
    &\times\left(\frac{1}{|G|}\sum_{\substack{f_p\in \Hom(G_{\Q_p},G)\\f_p(I_p)\subseteq H}} |\disc(f_p)|^{-s}\right)
\end{align*}
Corollary \ref{cor:meromorphic_continuation} implies that the infinite products given here are convergent for ${\rm Re}(s) \ge 1/d$.

We partition the poset of subgroups of $G$ into two separate posets: the poset of subgroups containing $C_2$, and the poset of subgroups not containing $C_2$. We address each of the two sums in the way we proved part (ii). Define $f_0:\{C_2\le H\le G\}\times\N \to \R_{\ge 0}$ to be multiplicative with $f_0(J,p^j)=0$ for $j>1$, $f_0(1,p)=0$, and
\[
f_0(J,p) = \left(\prod_{\substack{g\in G-\{1\}\\\ind(g) < d}}\left(1 - p^{-\ind(g)sf_p(\Q(\zeta_{|\langle g\rangle|})/\Q)}\right)^{\frac{b_{\ind(g)}(\Q,G)g_p(\Q(\zeta_{|\langle g\rangle|})/\Q)}{[\Q(\zeta_{|\langle g\rangle|}):\Q]}}\right)\left(\frac{1}{|G|}\sum_{\substack{f_p\in \Hom(G_{\Q_p},G)\\f_p(I_p)C_2=J}} |\disc(f_p)|^{-s}\right)
\]
otherwise. Define $f_1:\{H\le A\}\times\N \to \R_{\ge 0}$ to be multiplicative with $f_0(J,p^j)=0$ for $j>1$, $f_1(1,p)=0$, and
\[
f_1(J,p) = \left(\prod_{\substack{g\in G-\{1\}\\a(G) < \ind(g) < d}}\left(1 - p^{-\ind(g)sf_p(\Q(\zeta_{|\langle g\rangle|})/\Q)}\right)^{\frac{b_{\ind(g)}(\Q,G)g_p(\Q(\zeta_{|\langle g\rangle|})/\Q)}{[\Q(\zeta_{|\langle g\rangle|}):\Q]}}\right)\left(\frac{1}{|G|}\sum_{\substack{f_p\in \Hom(G_{\Q_p},G)\\f_p(I_p)=J}} |\disc(f_p)|^{-s}\right)
\]
otherwise. Taking these together, the limit
\[
\lim_{s\to 1/d} D_{\Q,G}(s)\prod_{\substack{g\in G-\{1\}\\a(G) < \ind(g) < d}}\zeta_{\Q(\zeta_{|\langle g\rangle|})}(\ind(g)s)^{-\frac{b_{\ind(g)}(\Q,G)}{[\Q(\zeta_{|\langle g\rangle|}):\Q]}}
\]
exists by Theorem \ref{thm:generating_mero}, and is equal to
\begin{align*}
\zeta\left(\frac{a(G)}{d}\right)\sum_{\substack{H\le G\\C_2\subseteq H}}\mu_G(H,G)\prod_{p<\infty}\left(1 + \sum_{J\le H}f_0(J,p)\right) + \sum_{\substack{H\le G\\C_2\not\subseteq H}}\mu_G(H,G)\prod_{p<\infty}\left(1 + \sum_{J\le H}f_0(J,p)\right).
\end{align*}
Lemma \ref{lem:poset_nonvanishing} implies that
\begin{align*}
    \sum_{\substack{H\le G\\C_2\subseteq H}}\mu_G(H,G)\prod_{p<\infty}\left(1 + \sum_{J\le H}f_0(J,p)\right) &= \sum_{n=2}^{\infty}\mu(n)^2 \prod_{\substack{(J_p)_{p\mid n}\\\prod_p J_p = G}} \prod_{p\mid n} f_0(J_p,p)\\
    & > 0
\end{align*}
by isolating a single nonzero term in the same way as the proof of part (ii). Similarly, the structure $G=C_2\times A$ for $|A|$ odd implies
\begin{align*}
    \sum_{\substack{H\le G\\C_2\not\subseteq H}}\mu_G(H,G)\prod_{p<\infty}\left(1 + \sum_{J\le H}f_1(J,p)\right) &= -\sum_{\substack{H\le A}}\mu_G(H,A)\prod_{p<\infty}\left(1 + \sum_{J\le A}f_1(J,p)\right)\\
    &=-\sum_{n=2}^{\infty}\mu(n)^2 \prod_{\substack{(J_p)_{p\mid n}\\\prod_p J_p = A}} \prod_{p\mid n} f_1(J_p,p)\\
    & < 0
\end{align*}
by isolating a single term in the same way as the proof of part (ii).

Finally, we know that $a(g)/d\in [0,1)$ implies that $\zeta(a(G)/d) < 0$, so that the limit is a sum of two negative values
\begin{align*}
\zeta\left(\frac{a(G)}{d}\right)\sum_{\substack{H\le G\\C_2\subseteq H}}\mu_G(H,G)\prod_{p<\infty}\left(1 + \sum_{J\le H}f_0(J,p)\right) + \sum_{\substack{H\le G\\C_2\not\subseteq H}}\mu_G(H,G)\prod_{p<\infty}\left(1 + \sum_{J\le H}f_0(J,p)\right) < 0.
\end{align*}
This concludes the proof.
\end{proof}

\subsection{Proof of Corollary \ref{cor:C4andC6}}

We first consider $C_{4M}$. When $(3,M)=1$, the minimum index $a(C_{4M}) = 2M$ is given by an element of order $2$. The next smallest index is $3M$ and is given by an element of order $4$. Theorem \ref{thm:nonvanishing}(ii) shows that $D_{\Q,C_{4M}}(s)$ has a pole of order $b_{3M}(\Q,C_{4M}) = 1$ at $s=\frac{1}{3M}$. Thus, it suffices to prove that $\theta(\Q,C_{4M})<\frac{1}{3M}$.

We apply Corollary \ref{cor:uncond} with $D = \frac{\ell}{4M(\ell-1)}$ for $\ell$ the smallest prime dividing $M$, which is the reciprocal of the index of an element of order $\ell$. This gives
\begin{align*}
\theta(\Q,C_{4M}) &\le \frac{1}{2M} - \frac{\frac{1}{2M} - \frac{\ell}{4M(\ell-1)}}{1 + \frac{1}{3}\left(1 - \frac{2M\ell}{4M(\ell-1)}\right) + \frac{2}{3}\left(1 - \frac{3M\ell}{4M(\ell-1)}\right)}=\frac{5}{16M} + \frac{3}{32M\ell - 48M}.\\
\end{align*}
This is smaller than $\frac{1}{3M}$ whenever
\[
\frac{3}{32\ell - 48} < \frac{11}{16}.
\]
This is true if and only if $\ell > 18/11$, and so is certainly true for $\ell \ge 5$.

We next consider $C_{6M}$. The minimum index $a(C_{6M}) = 3M$ is given by an element of order $2$. The next smallest index is $4M$ and is given by an element of order $3$. Theorem \ref{thm:nonvanishing}(iii) shows that $D_{\Q,C_{6M}}(s)$ has a pole of order $b_{4M}(\Q,C_{6M}) = 1$ at $s=\frac{1}{4M}$. Thus, it suffices to prove that $\theta(\Q,C_{6M})<\frac{1}{4M}$.

Suppose first that $5\nmid M$, so that $(30,M) = 1$. We apply Corollary \ref{cor:uncond} with $D = \frac{1}{5M}$ being the reciprocal of index of an element of order $6$, the next smallest order if $(30,M)=1$. This gives
\begin{align*}
\theta(\Q,C_{6M}) &\le \frac{1}{3M} - \frac{\frac{1}{3M} - \frac{1}{5M}}{1 + \frac{1}{3}\left(1 - \frac{3M}{5M}\right) + \frac{2}{3}\left(1 - \frac{4M}{5M}\right)} = \frac{13}{57M} < \frac{1}{4M}.
\end{align*}
If we have $5\mid M$ with $(6,M)=1$, let $D = \frac{5}{24M}$ be the reciprocal of the index of an element of order $5$. Then
\begin{align*}
\theta(\Q,C_{6M}) &\le \frac{1}{3M} - \frac{\frac{1}{3M} - \frac{5}{24M}}{1 + \frac{1}{3}\left(1 - \frac{3M}{24M/5}\right) + \frac{2}{3}\left(1 - \frac{4M}{24M/5}\right)} = \frac{62}{267M} < \frac{1}{4M}.
\end{align*}

\appendix

\section{Restricted Local Conditions and Alternate Orderings}\label{app:local_cond_and_ordering}

The methods of this paper are readily applicable to the more generalized question of counting abelian extensions ordered by other admissible invariants and with restricted local conditions, as was done for the main term in \cite{alberts-odorney2021}. We restricted the main body of the paper to discriminants so as not to overwhelm the reader with notation. In particular, the results for restricted local conditions would be quite involved to state explicitly.

\subsection{Ordering by Admissible Invariants}

Suppose we fix a weight function ${\rm wt}:G\to \R_{\ge 0}$ for which ${\rm wt}(g) = 0$ if and only if $g=1$. We say that an invariant $\inv:\prod_p H^1(K_p,G)\to I_K$ is an invariant of weight ${\rm wt}$ if, for all but finitely many places, $\nu_p(\inv(f)) = {\rm wt}(f(\tau_p))$ for $\tau_p$ a generator of the tame inertia group at $p$. Such an invariant is certainly admissible as defined in \cite{alberts2021}, so that \cite{alberts-odorney2021} give the asymptotic growth rate of the function
\[
\#\Surj_{\inv}(G_K,T,\pi;X) = \#\{\pi\in Z^1(G_K,T(\pi)): f*\pi\text{ surjective,}\ |\inv(f)|\le X\}
\]
for any abelian $T\normal G$ and homomorphism $\pi:G_K\to G$ with $\pi(G_K)T=G$.

We may define the power saving invariant for an arbitrary admissible ordering in analogy with Definition \ref{def:theta} (also with an analogous twisted version).

\begin{definition}\label{def:theta_inv}
    Let $K$ be a number field, $G\subseteq S_n$ a transitive group, ${\rm inv}$ an admissible invariant, and $D_{K,G,{\rm inv}}(s)$ the generating series of $G$-extensions of $K$ ordered by ${\rm inv}$. Define $\theta_{\rm inv}(K,G)$ to be the infimum of all real numbers $\theta$ for which
    \begin{enumerate}[(a)]
        \item $D_{K,G,{\rm inv}}(s)$ has a meromorphic continuation to ${\rm Re}(s) > \theta$ with at most finitely many poles $s_1,s_2,...,s_r$, and
        \item The number of $G$-extensions of $K$ satisfies the bound
        \[
            \#\Surj_{\rm inv}(G_K,G;X) - \sum_{j=1}^r \underset{s=s_j}{\rm Res}\left(D_{K,G,{\rm inv}}(s)\frac{X^s}{s}\right) = O(X^{\theta}),
        \]
        where the implied constant may depend on $K$ and $G\subseteq S_n$.
    \end{enumerate}
\end{definition}

The discriminant has weight given by the index function. One might predict that the results of this paper are true for any invariant with a weight function, but this is not quite the case. What is true is the following:

\begin{theorem}\label{thm:other_invariants}
    Let ${\rm wt}:G\to \R_{\ge 0}$ be a weight function which is invariant under conjugation and invertible powers, i.e. ${\rm wt}(g)= {\rm wt}(hgh^{-1})$ for any $g,h\in G$ and ${\rm wt}(g) = {\rm wt}(g^k)$ for any $g\in G$ and $k\in \Z$ such that $\langle g\rangle = \langle g^k\rangle$.
    
    Then the main results of this paper are true for the discriminant replaced by an invariant of weight ${\rm wt}$, where the index function is replaced by ${\rm wt}$ in each expression.
\end{theorem}

The only real consideration is in Theorem \ref{thm:euler_factors}. In order for the proof to go through that the Euler factors of $\widehat{w_s}(h)$ are sums of ${\rm tr}(\rho_{h|_{\C o}}(\Fr_p))p^{-{\rm wt}(o)s}$, the weight function needs to factor through the Galois orbits of $\Hom((\Q/\Z)^*,T)$. The Galois action on $T$ is via conjugation inside of $G$, while the Galois action on $(\Q/\Z)^*$ is by invertible powers through the cyclotomic map. If ${\rm wt}$ is invariant under both of these actions, then the proof goes through.

This condition is a natural one. The inertia subgroups are only well-defined up to conjugation, so by asking ${\rm wt}$ to be invariant under conjugation we are really asking that the invariant is independent of the choice of embedding for the decomposition group $G_{K_p}\hookrightarrow G_K$. It is also the case that the tame inertia group does not have a canonical generator. Thus, asking that ${\rm wt}$ is invariant under invertible powers implies that the invariant is independent of the choice of generator of tame inertia. While it is still possible to fix some of these choices and ask a counting question, the results in this paper suggest that the corresponding generating series may not be as nice.

One example of particular interest is the product of ramified primes invariant, ${\rm ram}$. This is an invariant of weight ${\rm wt}(g) = 1$ whenever $g\ne 1$. Theorem \ref{thm:other_invariants} applied to Theorem \ref{thm:generating_mero}, Corollary \ref{cor:uncond}, and Corollary \ref{cor:GLH} implies the following:

\begin{corollary}
    Let $K$ be a number field, $G$ an abelian group, and ${\rm ram}$ the product of ramified primes ordering. Then
    \begin{enumerate}[(i)]
        \item the generating series $D_{K,G,{\rm ram}}(s)$ is meromorphic on ${\rm Re}(s)>0$, and in the right halfplane ${\rm Re}(s) > 1/2$ has only the pole at $s=1$ of order exactly
        \[
            b_{\rm ram}(K,G)=\#\{\text{ orbits of }G-\{1\}\text{ under }\chi\},
        \]
        for $\chi:G_K\to \hat{Z}^{\times}$ the cyclotomic character.
        \item $\theta_{\rm ram}(K,G) \le 1 - \frac{3}{6+[K:\Q](|G|-1)}$, and
        \item the generalized Lindel\"of Hypothesis implies that $\theta_{\rm ram}(K,G) \le 1/2$.
    \end{enumerate}
\end{corollary}

This gives an unconditional power savings for each abelian $G$ of the form
\[
\#\Surj_{\rm ram}(G_K,G;X) = X P_{K,G,{\rm ram}}(\log X) + O\left(X^{1 - \frac{3}{6+[K:\Q](|G|-1)}+\epsilon}\right).
\]
The generalized Linedl\"of hypothesis improves this to $O(X^{1/2+\epsilon})$.

\subsection{Restricted Local Conditions}

The fact that Theorem \ref{thm:euler_factors} applies to $\widehat{w_s}(h)$ for \emph{any} $h\in H^1(K,T(\pi)^*)$ implies that we can use it to produce meromorphic continuations for any finite sum of $\widehat{w_s}(h)$. It is proven in \cite{alberts-odorney2021} that for a family of local conditions $\mathcal{L} = (L_p)$ for $L_p\subseteq H^1(K_p,T(\pi))$ satisfying
\begin{itemize}
    \item $p\mapsto L_p$ is Frobenian
    \item For all but finitely many $p$, $L_p$ is a union of coests of $H_{ur}^1(K_p,T(\pi))$ containing the identity coset,
\end{itemize}
the generating Dirichlet series for
\[
H_{\mathcal{L}}^1(K,T(\pi)) = H^1(K,T(\pi)) \cap \prod_p L_p
\]
ordered by an admissible invariant is a sum of finitely many of the $\widehat{w_{\mathcal{L},s}}(h)$.

It is possible to keep a close track of the Euler factors of $\widehat{w_{\mathcal{L},s}}(h)$ to prove a version of Corollary \ref{cor:meromorphic_continuation} where the orbits in the product is determined by which ramification types are allowed by $\mathcal{L}$. However, it is possible that, if $p\mapsto L_p$ is determined by the splitting type of $p$ in a nontrivial extension, the analogous statement to Corollary \ref{cor:meromorphic_continuation} will include non-integer rational powers of Artin $L$-functions. This gives a meromorphic continuation of the generating series to a \emph{branch cut} of the complex plane.

In order for Theorem \ref{thm:xi+1} to produce a power savings, we need an honest meromorphic continuation which means we would need to know which $\mathcal{L}$ give only integer powers of Artin $L$-functions in the analogous decomposition to Corollary \ref{cor:meromorphic_continuation}. Giving a complete classification would be to cumbersome for this appendix, but it is clear that at least the following is true:

\begin{theorem}
    Let $\mathcal{L} = (L_p)$ be a family of local conditions $L_p\subseteq H^1(K_p,T(\pi))$ such that there exists a subset $H\subseteq G$ closed under conjugation and invertible powers with $1\subsetneq H$ for which
    \[
    L_p = \{f\in H^1(K_p,T(\pi)) : f(I_p)\subseteq H\}
    \]
    for all but finitely many primes $p$.

    Then the counting function
    \[
        \#\Surj_{\mathcal{L}}(G_K,T,\pi;X) = \#\left\{f\in Z^1(G_K,T(\pi)) : f*\pi\text{ surjective},\ f\in \prod_p L_p,\ |\disc(f*\pi)| \le X\right\}
    \]
    satisfies the main results of this paper, where the definitions of $a(T)$, $b_d(K,T(\pi))$ are modified to give $a_H(T)$ and $b_{H,d}(K,T(\pi))$ by only considering the indices of $g\in H-\{1\}$, and the analogous power saving bounds $\theta_H(K,T)$ has unconditional bounds analogous to Corollary \ref{cor:uncond} and \ref{cor:uncond_twist} with the sum in the denominator over only those $g\in H-\{1\}$ with $\ind(g) < D$.
\end{theorem}

Checking that the lower order terms, or even the main term, do not vanish becomes significantly more involved when we restrict to a family of local conditions. Not only does the work in Section \ref{sec:examples_lower_order} become more difficult to perform, but we now have to worry about Grunwald--Wang counter examples. For the main term when $T=G$ has the trivial action this is considered by Wood \cite{wood2009}, whereas for $T\ne G$ this was originally overlooked in \cite{alberts-odorney2021} and was corrected in a follow-up Corrigendum \cite{alberts-odorney2023}. Essentially, there are certain local conditions that are never achievable which causes not only the main term, but the entire generating Dirichlet series to vanish.

An interesting question is whether there exist weaker Grunwald--Wang counter examples for which the main term does not vanish, but subsequent lower order terms do. In other words, are there any cases where the analogous statement to Conjecture \ref{conj:Q_order_of_poles} for restricted local conditions is true for the rightmost pole but false for another pole?

\bibliographystyle{alpha}

\bibliography{Power_Savings_for_Abelian_Extensions.bbl}

\end{document}